\documentclass[aip, cha, preprint]{revtex4-1}

\usepackage{graphicx}
\usepackage{dcolumn}
\usepackage{bm}

\usepackage[utf8]{inputenc}
\usepackage[T1]{fontenc}
\usepackage{mathptmx}
\usepackage{etoolbox}

\usepackage{bm}\usepackage{soul}
\usepackage{empheq}\usepackage{mathtools}
\usepackage[top=0.6in, bottom=1in, left=1in, right=0.8in]{geometry}
\usepackage{mathrsfs,color}
\usepackage{amsfonts}\usepackage{multirow}
\usepackage{amssymb}\usepackage{comment}
\usepackage{amsmath}
\usepackage{amssymb,amsthm}
\usepackage{graphicx,afterpage,cancel,physics}
\def\XXint#1#2#3{{\setbox0=\hbox{$#1{#2#3}{\int}$ }
\vcenter{\hbox{$#2#3$ }}\kern-.6\wd0}}

\newtheorem{thm}{Theorem}[section]
\newtheorem{corollary}[thm]{Corollary}

\newtheorem{proposition}[thm]{Proposition}

\definecolor{orange}{RGB}{255,127,0}
\definecolor{gray}{RGB}{128,128,128}

\def\d{{\, \rm d}}

\afterpage{\clearpage}

\newcommand{\KL}{D_{\mathrm{KL}}}

\def\d{{\, \rm d}}

\begin{document}




\title[Mechanisms and pathways of extreme events in partially observed stochastic dynamical systems]{Mechanisms and pathways of extreme events in partially observed stochastic dynamical systems}
\author{Charlotte Moser}
\author{Nan Chen}%
 \email{chennan@math.wisc.edu}
 \author{Marios Andreou}
 \affiliation{Mathematics Department, University of Wisconsin-Madison, Madison, Wisconsin 53705, USA.}

\date{November 2026}

\begin{abstract}
Extreme events occur across the natural, engineering, and socioeconomic sciences, where rare but high-impact episodes can lead to disproportionate consequences that pose major challenges for prediction and risk management. Existing studies have mainly focused on the statistics, sampling, forecasting, and attribution of extremes from observable variables. In this paper, we develop a mathematical framework for studying the mechanisms and pathways of extreme events in partially observed stochastic dynamical systems with hidden variables. By integrating data assimilation with information-theoretic and trajectory-based diagnostics, we infer latent precursor dynamics from observations, quantify their uncertainty, and determine how their influence propagates toward an observed extreme event. Conditional Gaussian models provide a tractable analytical setting for deriving closed-form diagnostics, while the framework extends more broadly through numerical methods. The analysis proceeds from two complementary perspectives. From a trajectory-wise viewpoint, we compare filtering and smoothing distributions to identify the onset of hidden precursors and quantify their temporal influence range. From a statistical viewpoint, we construct event-conditioned hidden-state distributions to identify sensitive triggering directions, representative latent pathways, and distinct candidate pathway classes through clustering. Three case studies illustrate the methodology. In an intermittent stochastic model, where hidden damping dynamics emerge before observed bursts, discrepancies between the filter and smoother provide an onset diagnostic. In a stochastic model with damping and forcing, separate damping-induced, forcing-driven, and mixed pathways to extremes are identified. In a nonlinear topographic-flow model, distinct mechanisms and pathways for blocking and unblocking patterns associated with the observed extreme events are revealed.
\end{abstract}

\maketitle 

\textbf{Extreme events such as severe weather, market crashes, sudden state transitions, and structural engineering failures can have drastic repercussions. Thus, it is essential to understand their underlying triggers, which are often driven by mechanisms that cannot be directly observed. Many existing approaches rely solely on observations for characterizing and predicting extreme events, while the hidden dynamical processes that trigger them remain much less understood. This work develops a mathematical framework to uncover and analyze the hidden dynamics that drive extremes in partially observed random systems. By integrating data, models, information-theory diagnostics, and trajectory-based analysis, the framework exposes how hidden variables evolve prior to an extreme event, quantifies the uncertainty in these hidden pathways, and determines how their influence propagates through the system. The methodology reveals both the mechanisms underlying individual trajectories and representative statistics for different types of extreme event pathways. Applications to three stochastic dynamical models demonstrate how distinct hidden mechanisms of extremes can be systematically identified even when only partial observations are available. Overall, the framework provides new tools for understanding the origins of extreme events and improving their interpretation in complex systems with incomplete observations and uncertainty.}

\section{Introduction}

Extreme events, often associated with large amplitudes, are generally characterized by rare occurrences that lie in the tails of a system's statistical distribution \cite{coles2001introduction, lucarini2016extremes}. They often have the potential to cause catastrophic impacts and induce substantial scientific and societal consequences \cite{albeverio2006extreme}. Extreme events arise across a wide range of spatiotemporal scales. For example, an extreme El Ni\~no event can trigger numerous localized extremes, including severe droughts, intense heatwaves, enhanced wildfire risk, and increased hurricane activity \cite{santoso2017defining}. Beyond Earth science, extreme events also appear in engineering and applied disciplines, including structural failures in materials science \cite{zhang2025physics}, optical and oceanic rogue waves \cite{akhmediev2016roadmap, kharif2008rogue}, black-swan shocks in financial markets \cite{phadnis2021study}, and rare transient events in excitable systems \cite{ansmann2013extreme}. Their mechanisms, impacts, and attribution across disciplines have been widely studied in Refs.~\cite{farazmand2019extreme, trenberth2015attribution, noy2024extreme}, while their statistical behavior and associated mathematical challenges are reviewed in Refs.~\cite{sapsis2021statistics, ghil2011extreme}.

Recent progress in extreme-event research has advanced along several complementary directions. Efficient sampling of rare and extreme events has important practical significance \cite{mohamad2018sequential, mackay2021sampling, finkel2026rare}, with a variety of rare-event algorithms being developed and applied to Earth system problems \cite{webber2019practical, finkel2024bringing}. Forecasting extreme events using both dynamical-system approaches \cite{kaveh2025spatiotemporal, chen2020predicting} and machine-learning methods \cite{chang2025extreme, guth2019machine, farazmand2017variational, sun2025can, mojgani2023extreme, guan2026prediction} has also become increasingly active, including settings with limited training data. In addition, significant effort has been devoted to understanding the mechanisms underlying extreme events from a dynamical system perspective \cite{farazmand2019extreme, alvre2024studying, chowdhury2022extreme, mishra2020routes, farazmand2017variational, lucarini2016extremes}. These developments have substantially improved our ability to characterize, sample, and predict extremes.

Observed extreme events are often associated with intermittent behavior driven by the underlying nonlinear dynamics and instabilities \cite{lovejoy2018spectra, overland2021intermittency}. Yet the mechanisms that generate these extremes are typically complex and not directly inferable from observations alone. Statistical indicators of extremes, extracted solely from observations, can only provide corroborative signatures rather than predictive or necessary conditions \cite{brovkin2021past, kuehn2011mathematical, thomas2016using}. In many realistic systems, the governing dynamics involve hidden or unresolved processes that interact nonlinearly with the observed variables \cite{chen2023stochastic, majda2018model, grigorian2025learning, boyen2013discovering}. Extreme events frequently emerge from these multiscale interactions, in which small-scale or unobserved modes, via nonlinear and stochastic coupling, amplify through instability the large-scale observed components and ultimately trigger the visible extreme event \cite{majda2005information}. Such latent precursors are inherently hidden in purely observation-based analysis. Understanding how unresolved components initiate extreme events, identifying their precursors, and quantifying their influence on the timing and duration of extreme episodes are therefore both scientifically important and practically valuable. This dynamical characterization can improve mechanism-level understanding and attribution of extremes, as well as enhance the accuracy and lead time of early warning and prediction.

In this paper, we develop a mathematical framework for studying the mechanisms and pathways of extreme events in partially observed stochastic systems with hidden variables. The central idea is to integrate data assimilation with information-theoretic and trajectory-based diagnostics in order to recover latent precursor dynamics, quantify their uncertainty, and determine how their influence propagates toward an observed extreme event. We focus on settings in which the hidden variables represent small-scale or unresolved latent processes that interact with the observed components and may begin to emerge and evolve well before an observable extreme event occurs. This makes them natural carriers of hidden precursor information. Inferring these hidden states from observations, quantifying the associated uncertainty, and determining how their effects propagate forward in time toward the observed extreme event will provide essential information about event precursors. Addressing these questions is particularly important in multiscale systems, where hidden dynamics may affect not only whether an extreme event occurs, but also when it begins, how rapidly it develops, and how long its impact persists. For general nonlinear systems, hidden-state inference is rarely available in closed form and typically requires computational approximations. To enable rigorous analysis, we consider a dynamically rich and practically relevant class of nonlinear stochastic dynamical systems known as conditional Gaussian nonlinear systems \cite{chen2018conditional, liptser2001statistics}. These systems admit analytically tractable solutions for data assimilation and event-conditioned statistics, allowing the proposed diagnostics to be examined without contamination from approximation error. Nevertheless, the methodology developed here does not rely on conditional Gaussianity. For more general nonlinear systems, the same diagnostics can be implemented through ensemble- or particle-based methods, which provide approximate reconstructions of hidden trajectories from partial observations.

The analysis proceeds from two complementary viewpoints. The first is a trajectory-wise approach that focuses on individual extreme events. For each event, we reconstruct the hidden variables along the evolution leading to the observed extreme event and thereby uncover the specific dynamical pathway through which that event arises. This allows event-wise attribution and provides a natural way to identify precursor timing and influence duration. The second is a statistical methodology that considers extreme events collectively. By conditioning on the occurrence of observed extremes and aggregating reconstructions of the hidden variables across many events, we obtain the corresponding hidden-state distribution during extreme episodes. This perspective reveals common precursor structures shared across events, identifies sensitive directions in the hidden state space associated with triggering mechanisms, and constructs representative latent pathways. Through classification, the combined insights from both viewpoints help distinguish different pathway classes and assess whether they arise from distinct underlying dynamical mechanisms. The overall structure of the proposed framework is summarized in Figure~\ref{fig:overview}.

\begin{figure}[t]
\centering
\includegraphics[width=\textwidth]{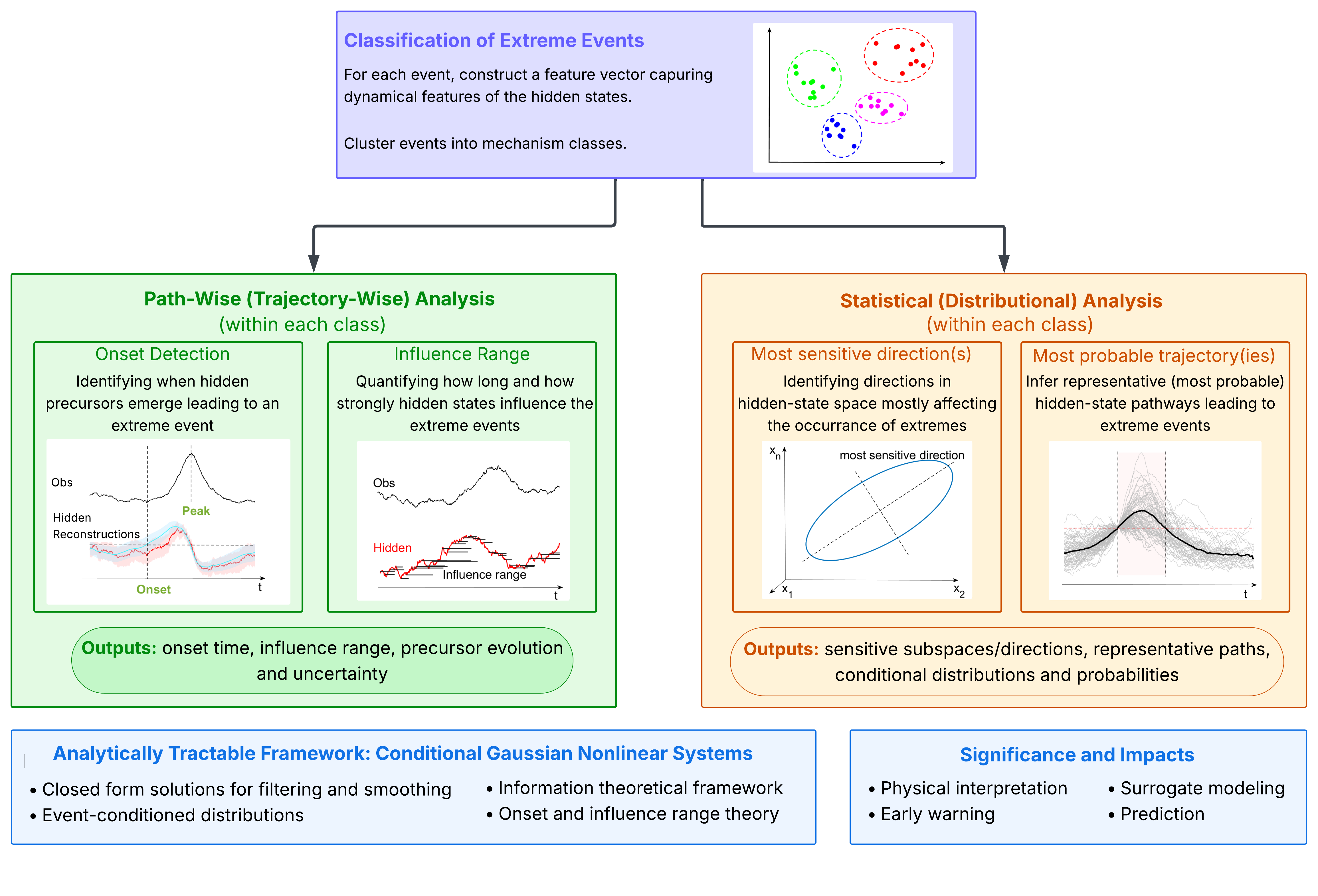}
\caption{
Schematic overview of the proposed framework for diagnosing hidden mechanisms and pathways of extreme events in partially observed stochastic dynamical systems. }
\label{fig:overview}
\end{figure}

The rest of the paper is organized as follows. Section \ref{Sec:Framework} introduces the general hidden-state inference setting, presents conditional Gaussian systems as an analytically tractable subclass, and discusses both closed-form and ensemble-based strategies. Section \ref{Sec:Mechanisms} develops trajectory-wise, statistical, and classification tools for analyzing the mechanisms and pathways underlying observed extreme events. Section \ref{Sec:Numerics} demonstrates the proposed framework in several prototype nonlinear stochastic dynamical systems. The paper is concluded in Section \ref{Sec:Conclusion}.

\section{Hidden-State Inference Framework for Extreme-Event Diagnostics}\label{Sec:Framework}
\subsection{Problem setup}
Consider a coupled nonlinear stochastic system \cite{crisan2011oxford, rozovskii2012stochastic},
\begin{subequations}\label{General_System}
\begin{align}
  \frac{\d\mathbf{X}}{\d t} &= \mathbf{F}(\mathbf{X},\mathbf{Y},t) + \mathbf{B}_\mathbf{1}(\mathbf{X},\mathbf{Y},t)\dot{\mathbf{W}}_\mathbf{1}(t),\label{General_X}\\
  \frac{\d\mathbf{Y}}{\d t} &= \mathbf{G}(\mathbf{X},\mathbf{Y},t) + \mathbf{b}_\mathbf{2}(\mathbf{X},\mathbf{Y},t)\dot{\mathbf{W}}_\mathbf{2}(t),\label{General_Y}
\end{align}
\end{subequations}
where $\mathbf{X}\in\mathbb{R}^{n_\mathbf{X}}$ and $\mathbf{Y}\in\mathbb{R}^{n_\mathbf{Y}}$ denote multi-dimensional state variables. The drift terms $\mathbf{F}$ and $\mathbf{G}$ are general nonlinear functions of the joint state and time $t$, while $\mathbf{B}_1$ and $\mathbf{b}_2$ are state-dependent noise coefficient matrices. The processes $\dot{\mathbf{W}}_1$ and $\dot{\mathbf{W}}_2$ are independent mean-zero Gaussian random vectors with uncorrelated components that possess unit variance.

In many practical settings, only a subset of the state variables is directly observed. Here, $\mathbf{X}$ represents the observed components, while $\mathbf{Y}$ denotes the unobserved (hidden) variables. This separation can arise from measurement limitations, for example, satellite observations of sea surface quantities versus unobserved subsurface ocean dynamics. It may also come from modeling considerations, where $\mathbf{X}$ corresponds to resolved large-scale variables and $\mathbf{Y}$ represents unresolved or subgrid-scale processes.

Extreme events are naturally identified in the observed variables $\mathbf{X}$. However, the mechanisms responsible for triggering, amplifying, and sustaining these events often involve interactions with the hidden variables $\mathbf{Y}$. Understanding how the dynamics of $\mathbf{Y}$ influence the occurrence of extreme events in $\mathbf{X}$ is therefore a central challenge. This problem is particularly difficult due to the nonlinear coupling between $\mathbf{X}$ and $\mathbf{Y}$ and the presence of stochastic forcing.

The goal of this paper is to develop diagnostics for hidden mechanisms in partially observed nonlinear systems. In general models, conditional distributions of the hidden variables given observations are rarely available in closed form and must be approximated through numerical data assimilation methods. We therefore distinguish between the general framework and special model classes, such as conditional Gaussian systems, where the inference problem becomes analytically tractable.

\subsection{Conditional Gaussian nonlinear stochastic systems}
We now introduce conditional Gaussian nonlinear systems, an important subclass of partially observed nonlinear stochastic models. We use this class as an analytically tractable setting for exact hidden-state inference \cite{liptser2001statistics, chen2018conditional}, while the proposed diagnostic framework itself is not restricted to conditional Gaussian systems and can be extended to general partially observed nonlinear systems through approximate data assimilation. Conditional Gaussian nonlinear systems have the form, 
\begin{subequations}\label{CGNS}
\begin{align}
  \frac{\d\mathbf{X}}{\d t} &= \mathbf{A}_0(\mathbf{X},t) + \mathbf{A}_1(\mathbf{X},t)\mathbf{Y} + \mathbf{B}_1(\mathbf{X},t)\dot{\mathbf{W}}_1(t),\label{CGNS_X}\\
  \frac{\d\mathbf{Y}}{\d t} &=\mathbf{a}_0(\mathbf{X},t) + \mathbf{a}_1(\mathbf{X},t)\mathbf{Y} + \mathbf{b}_2(\mathbf{X},t)\dot{\mathbf{W}}_2(t),\label{CGNS_Y}
\end{align}
\end{subequations}
 where $\mathbf{A}_0$, $\mathbf{a}_0$, $\mathbf{A}_1$, $\mathbf{a}_1$, $\mathbf{B}_1$, and $\mathbf{b}_2$ are vector- or matrix-valued functions that may depend nonlinearly on the observed state $\mathbf{X}$ and time $t$.

A key feature of \eqref{CGNS} is that the hidden variables $\mathbf{Y}$ enter only linearly in the system and only appear in the drift terms, while the coefficients depend nonlinearly on $\mathbf{X}$. Despite this conditional linearity, the coupled system remains highly nonlinear due to the nonlinear dependence on $\mathbf{X}$ and the multiplicative noise structure. As a result, the marginal distributions $p(\mathbf{X})$ and $p(\mathbf{Y})$, as well as the joint distribution $p(\mathbf{X},\mathbf{Y})$, can exhibit strong non-Gaussian features, including intermittency and heavy tails. Extreme events therefore arise naturally within this framework.

The importance of the system \eqref{CGNS} lies in its analytical tractability. In particular, conditioned on the trajectory of the observed variables $\mathbf{X}$, the distribution of the hidden variables $\mathbf{Y}$ remains Gaussian. This property enables closed-form characterization of the conditional dynamics, including filtering and smoothing distributions, which are generally unavailable for fully nonlinear systems. As will be shown below, this structure provides a foundation for decomposing the dynamics and attributing the mechanisms underlying extreme events.

The modeling framework in \eqref{CGNS} encompasses a wide range of classical and practically relevant systems that exhibit extreme behavior across scientific disciplines \cite{chen2018conditional}. These include physics-constrained stochastic models, stochastically coupled reaction-diffusion systems in neuroscience and ecology, and low-order or reduced models in turbulence, fluid dynamics, and geophysical flows. Representative examples include stochastic versions of the Lorenz models, conceptual models for Charney-DeVore flows and atmospheric low-frequency variability, and reduced models for monsoon dynamics and El Ni\~no-Southern Oscillation. In addition, fundamental systems such as the Boussinesq equations can be interpreted within this framework when decomposed into resolved and unresolved components.

More broadly, in multiscale systems where nonlinear interactions are dominated by quadratic advection terms as in many fluid or geophysical fluid systems, it is natural to decompose the state into large-scale variables $\mathbf{X}$ and small-scale modes $\mathbf{Y}$. For the resulting reduced system to take the conditional Gaussian form in \eqref{CGNS}, the unresolved self-interactions must be parameterized so that the deterministic dynamics remain linear in $\mathbf{Y}$. Higher-order interactions among the small-scale modes can, for example, be modeled stochastically to effectively capture their aggregate statistical effects, while the nonlinear coupling between the large- and small-scale variables is retained \cite{majda2006nonlinear, majda2005information}. This leads naturally to systems of the form \eqref{CGNS}. These features make conditional Gaussian nonlinear systems a flexible and practically relevant framework for studying the mechanisms and pathways of extreme events in complex dynamical systems.

\subsection{Data assimilation in conditional Gaussian systems: filtering and smoothing}

 The diagnostics developed later require estimates of the conditional distributions of hidden variables given partial observations. In general nonlinear systems, these distributions are obtained approximately through numerical data assimilation methods, including ensemble Kalman filters, particle filters, and related smoothing approaches \cite{asch2016data, law2015data, sarkka2023bayesian}. However, for conditional Gaussian systems, the filtering and smoothing distributions remain Gaussian and satisfy closed-form evolution equations. We summarize this exact setting below, which will be used extensively in the first two numerical examples. In what follows, for simplicity, the observed time series of $\mathbf{X}$ is assumed to be continuously observed over $[0,T]$, which we condition on by using the notation $\cdot\mid\mathbf{X}_{0:\tau}$, $\tau\in[0,T]$; an analogous framework can be developed for discrete-in-time observations \cite{liptser2001statistics}.

A key feature of the system \eqref{CGNS} is that, given a realization of the observed trajectory $\mathbf{X}(s)$ for $s \le t$, the hidden variable $\mathbf{Y}(t)$ becomes conditionally linear and Gaussian \cite{liptser2001statistics}. This formally follows from substituting the known trajectory $\mathbf{X}(s)$ into \eqref{CGNS}, which reduces the dynamics of $\mathbf{Y}$ to a linear stochastic differential equation with time-dependent but known coefficients.

As a result, the conditional distribution
\begin{equation}\label{CGNS_PDF}
    p(\mathbf{Y}(t)\mid \mathbf{X}_{0:t}) \sim \mathcal{N}(\boldsymbol\mu_{\mathbf{f}}(t), \mathbf{R}_{\mathbf{f}}(t))
\end{equation}
is Gaussian. The conditional mean $\boldsymbol\mu_{\mathbf{f}}(t)$ and covariance $\mathbf{R}_{\mathbf{f}}(t)$ satisfy the following closed-form equations \cite{liptser2001statistics}:
\begin{subequations}\label{CGNS_Stat}
\begin{align}
  \frac{\d \boldsymbol{\mu}_{\mathbf{f}}}{\d t} &= (\mathbf{a}_0 + \mathbf{a}_1 \boldsymbol{\mu}_{\mathbf{f}})\, 
  + (\mathbf{R}_{\mathbf{f}}\mathbf{A}_1^\mathtt{T})(\mathbf{B}_1\mathbf{B}_1^\mathtt{T})^{-1}
  \left(\frac{\d\mathbf{X}}{\d t} - (\mathbf{A}_0 + \mathbf{A}_1\boldsymbol{\mu}_{\mathbf{f}})\right),\label{CGNS_Stat_Mean}\\
  \frac{\d\mathbf{R}_{\mathbf{f}}}{\d t} &= \mathbf{a}_1 \mathbf{R}_{\mathbf{f}} + \mathbf{R}_{\mathbf{f}}\mathbf{a}_1^\mathtt{T} + \mathbf{b}_2\mathbf{b}_2^\mathtt{T}
  - (\mathbf{R}_{\mathbf{f}}\mathbf{A}_1^\mathtt{T})(\mathbf{B}_1\mathbf{B}_1^\mathtt{T})^{-1}(\mathbf{A}_1\mathbf{R}_{\mathbf{f}}).\label{CGNS_Stat_Cov}
\end{align}
\end{subequations}
Here, $\cdot^\mathtt{T}$ denotes the matrix transpose. The filter equations \eqref{CGNS_Stat} are random differential equations due to their dependence on the observed process $\mathbf{X}$, with the covariance equation \eqref{CGNS_Stat_Cov} being a random continuous-time Riccati equation. The conditional distribution $p(\mathbf{Y}(t)\mid \mathbf{X}(s),\, s \le t)$ corresponds to the posterior distribution in Bayesian data assimilation, specifically the filtering solution. Accordingly, $\boldsymbol\mu_{\mathbf{f}}$ and $\mathbf{R}_{\mathbf{f}}$ are referred to as the posterior filter mean and posterior filter covariance, respectively. The classical Kalman-Bucy filter \cite{kalman1961new} arises as a special case of \eqref{CGNS_Stat} when the system is fully linear with constant or time-dependent coefficients.

Filtering provides an online estimate of the hidden state using observations up to the current time. However, for retrospective analysis, incorporating future observations can significantly improve state estimation. This leads to the smoothing problem \cite{sarkka2023bayesian}. Given a full observation trajectory $\mathbf{X}(t)$ over $t \in [0,T]$, the smoother distribution
\begin{equation}\label{Smoother}
  p(\mathbf{Y}(t)\mid \mathbf{X}_{0:T}) \sim \mathcal{N}(\boldsymbol\mu_\mathbf{s}(t), \mathbf{R}_\mathbf{s}(t))
\end{equation}
remains Gaussian \cite{chen2020efficient}. The smoother mean $\boldsymbol\mu_\mathbf{s}(t)$ and covariance $\mathbf{R}_\mathbf{s}(t)$ satisfy the backward equations
\begin{subequations}\label{Smoother_Main}
\begin{align}
  \frac{\overleftarrow{\d \boldsymbol{\mu}_\mathbf{s}}}{\d t} &=
  -\mathbf{a}_0 - \mathbf{a}_1\boldsymbol{\mu}_\mathbf{s}
  + (\mathbf{b}_2\mathbf{b}_2^\mathtt{T})\mathbf{R}_{\mathbf{f}}^{-1}(\boldsymbol\mu_{\mathbf{f}} - \boldsymbol{\mu}_\mathbf{s}),\label{Smoother_Main_mu}\\
  \frac{\overleftarrow{\d \mathbf{R}_\mathbf{s}}}{\d t} &=
  -\left(\mathbf{a}_1 + (\mathbf{b}_2\mathbf{b}_2^\mathtt{T}) \mathbf{R}_{\mathbf{f}}^{-1}\right)\mathbf{R}_\mathbf{s}
  - \mathbf{R}_\mathbf{s}\left(\mathbf{a}_1^\mathtt{T} + \mathbf{R}_{\mathbf{f}}^{-1}(\mathbf{b}_2\mathbf{b}_2^\mathtt{T})\right)
  + \mathbf{b}_2\mathbf{b}_2^\mathtt{T},\label{Smoother_Main_R}
\end{align}
\end{subequations}
where $\boldsymbol\mu_{\mathbf{f}}$ and $\mathbf{R}_{\mathbf{f}}$ are obtained from \eqref{CGNS_Stat}. The backward differential $\overleftarrow{\d\cdot}$ indicates that \eqref{Smoother_Main} is solved backward in time over $[0,T]$, with terminal condition $(\boldsymbol\mu_\mathbf{s}(T), \mathbf{R}_\mathbf{s}(T)) = (\boldsymbol\mu_{\mathbf{f}}(T), \mathbf{R}_{\mathbf{f}}(T))$.

Although the explicit formulas above are specific to conditional Gaussian systems, the subsequent trajectory-wise diagnostics depend only on having filtering and smoothing approximations of the hidden-state distributions. Therefore, they can be directly transferred to general nonlinear systems once these distributions are computed numerically.

Beyond providing efficient state estimation, the filtering and smoothing distributions play two complementary roles in the present study. From a statistical perspective, the collection of conditional Gaussian posterior distributions associated with different realizations of the observed trajectories forms the basis for constructing mixture representations of the marginal distribution of the hidden variables. This provides a natural way to characterize the underlying geometry in the space associated with the hidden variables that triggers the observed extreme events. From a dynamical perspective, the conditional distribution describes the evolution of the hidden variables along individual trajectories, which enables a path-wise analysis of extreme events. In particular, filtering reflects the information available in real time, while smoothing incorporates future observations and thus provides a more complete reconstruction of the hidden-state dynamics leading to extreme events. This is because it can capture the influence of hidden precursors that begin emerging well before the observed extreme event and evolve over a long time period, which may not be fully revealed by the online filter. This distinction will be exploited below to assess whether the mechanisms underlying extreme events can be reliably identified from online information alone or require retrospective analysis using smoothing.

\subsection{Conditional mixture representations}

For general nonlinear systems, the joint distribution of the observed and hidden variables can be represented as a mixture of the conditional distributions of the hidden state given different realizations of the observed trajectory. This provides a direct link between trajectory-wise state estimation and statistical characterization of the hidden dynamics. In particular, it allows one to compare the full marginal distribution of the hidden variables with the portion associated with extreme events in the observed variables, thereby identifying the subset of hidden states that is most relevant for triggering extreme episodes.

In general, the conditional distribution is not necessarily Gaussian and often not analytically tractable. It must be estimated using ensemble approaches, which incorporate additional approximation error that depends heavily on the number of samples. Thus, approximating the joint distribution requires sampling not only different realizations of the observed trajectories, but sampling the conditional distribution as well, which can become prohibitively expensive in sparsely observed high-dimensional systems.

On the other hand, the conditional Gaussian nonlinear systems admit a powerful advantage in this context, as the conditional distribution of the hidden state given a realization of the observed state has a known closed form \eqref{CGNS_Stat} or \eqref{Smoother_Main}. Consequently, recovering the joint distribution of the observed and hidden states only requires sampling of the observed state \cite{chen2018efficient}. This dramatically reduces the sampling burden of high-dimensional hidden variables.

In either the general or conditional Gaussian case, using the filter or smoother distribution to form the conditional mixture provides distinct insights. The filtering representation, using only past and current information, is natural for online inference, while the smoothing representation, using the full observation interval, is better suited for retrospective analysis and event attribution. We state the smoothing version first, since it will be used below to analyze hidden-state distributions conditioned on extreme events. In the following propositions, the subscript of the expectations $\mathbb{E}[\cdot]$ indicates the distribution with respect to which they are taken.

\begin{proposition}[Conditional mixture representation]\label{prop:smoothing_mixture}
Let $(\mathbf{X}(t),\mathbf{Y}(t))$ be the solution of \eqref{General_System}, and assume that for each realization of the observed path $\mathbf{X}_{0:T}$ the smoothing distribution of $\mathbf{Y}(t)$ is $p_t(\mathbf y\mid \mathbf X_{0:T})$. Then the joint distribution of $(\mathbf{X}(t),\mathbf{Y}(t))$ at time $t$ can be written as
\begin{equation}\label{eq:smoothing_mixture}
p_t(\mathbf{x},\mathbf{y})
=
\int
\delta\big(\mathbf{x}-\mathbf{X}(t)\big)\,
p_t(\mathbf y\mid \mathbf X_{0:T})\,
\mathbb{P}(\mathrm{d}\mathbf{X}_{0:T}),
\end{equation}
where $\delta(\cdot)$ is the Dirac delta function, $\mathbf{X}(t)$ denotes the realization of the observed process at time t, and $\mathbb{P}(\mathrm{d}\mathbf{X}_{0:T}) = p(\mathbf{X}_{0:T}) \mathrm{d}\mathbf{X}_{0:T}$ is the probability measure over all possible sample paths of $\mathbf{X}$ in $[0,T]$.
The representation \eqref{eq:smoothing_mixture} is understood in the weak sense as: for any integrable (with respect to $p_t(x,y)$) and measurable test function $\varphi$,
\begin{equation}\label{eq:smoothing_test_function}
\mathbb{E}_{p_t(\mathbf{x},\mathbf{y})}\big[\varphi(t,\mathbf{X}(t),\mathbf{Y}(t))\big]
=
\mathbb{E}_{p(\mathbf{X}_{0:T})}\!\left[
\int
\varphi\big(t,\mathbf{X}(t),\mathbf{y}\big)\,
p_t(\mathbf y\mid \mathbf X_{0:T})\,
\mathrm{d}\mathbf{y}
\right].
\end{equation}
\end{proposition}

\begin{proof}[Proof of Proposition \ref{prop:smoothing_mixture}]
For any test function $\varphi(t,\mathbf{x},\mathbf{y})$ that is measurable and integrable with respect to $p_t(x,y)$, the tower property \cite{williams1991probability} gives
\begin{equation}
\mathbb E_{p_t(\mathbf{x},\mathbf{y})}[\varphi(t,\mathbf X(t),\mathbf Y(t))]
=
\mathbb E_{p(\mathbf{X}_{0:T})}\!\left[
\mathbb E_{p_t(\mathbf{y}|\mathbf{X}_{0:T})}\big[\varphi(t, \mathbf X(t),\mathbf Y(t))\mid \mathbf X_{0:T}\big]
\right].
\end{equation}
Conditioned on the observed path $\mathbf X_{0:T}$, the smoothing distribution of $\mathbf Y(t)$ is $p_t(\mathbf y\mid \mathbf X_{0:T})$. Therefore,
\begin{equation}
\mathbb E_{p_t(\mathbf y\mid \mathbf X_{0:T})}\big[\varphi(t,\mathbf X(t),\mathbf Y(t))\mid \mathbf X_{0:T}\big]
=
\int
\varphi\big(t,\mathbf X(t),\mathbf y\big)\,
p_t(\mathbf y\mid \mathbf X_{0:T})\,
\mathrm d\mathbf y.
\end{equation}
Substituting this into the previous identity yields
\begin{equation}
\mathbb E_{p_t(\mathbf{x},\mathbf{y})}[\varphi(t,\mathbf X(t),\mathbf Y(t))]
=
\mathbb E_{p(\mathbf{X}_{0:T})}\!\left[
\int
\varphi\big(t,\mathbf X(t),\mathbf y\big)\,
p_t(\mathbf y\mid \mathbf X_{0:T})\,
\mathrm d\mathbf y
\right],
\end{equation}
which is exactly \eqref{eq:smoothing_test_function}. It follows that \eqref{eq:smoothing_test_function} is the weak formulation of the measure-based representation in \eqref{eq:smoothing_mixture}.
\end{proof}

A completely analogous formula holds if the smoothing distribution $p_t(\mathbf y\mid \mathbf X_{0:T})$ is replaced with the filtering distribution $p_t(\mathbf y\mid \mathbf X_{0:t})$ and the conditioning is taken with respect to $\mathbf{X}_{0:t}$ rather than $\mathbf{X}_{0:T}$ \cite{chen2018efficient}. Thus, the same conditional mixture methodology can be used for both online and retrospective analyses. The corresponding formulas for the conditional Gaussian case are given by replacing the general conditional distribution $p_t(\mathbf y\mid \mathbf X_{0:T})$ with the Gaussian smoother $\mathcal N (\boldsymbol\mu_{\mathbf{s}},\mathbf{R}_{\mathbf{s}})$ or filter $\mathcal N  (\boldsymbol\mu_{\mathbf{f}},\mathbf{R}_{\mathbf{f}})$ distributions, and conditioning on the appropriate trajectory $\mathbf{X}_{0:T}$ or $\mathbf{X}_{0:t}$. 

In general nonlinear systems, $p_t(\mathbf y\mid \mathbf X_{0:T})$ is not explicitly known and must be approximated using ensemble or particle methods. In contrast, the conditional Gaussian structure \eqref{CGNS} provides a major computational advantage in the mixture representation. Since the dependence on the hidden variable is treated analytically within each Gaussian component, one can evaluate many quantities of interest explicitly, including moments, conditional expectations, and event-conditioned statistics, without direct sampling in the $\mathbf{Y}$-space required for general nonlinear systems. The only remaining sampling is over the distribution of the observed trajectories, $p(\mathbf{X}_{0:T})$ or $p(\mathbf{X}_{0:t})$.

To simplify notation, define the conditional expectation functional
\begin{equation}\label{eq:G_def}
G_t(\mathbf X_{0:T})
=
\mathbb E_{p_t(\mathbf y\mid \mathbf X_{0:T})}\big[\varphi(t,\mathbf X(t),\mathbf Y(t))\mid \mathbf X_{0:T}\big]
=
\int
\varphi\big(t,\mathbf X(t),\mathbf y\big)\,
p_t(\mathbf y\mid \mathbf X_{0:T})\,
\mathrm d\mathbf y.
\end{equation}

We only describe the smoothing version of the following proposition, but the same applies to the filtering with the appropriate conditioning.

\begin{proposition}[Monte Carlo accuracy and hidden-dimension independence]\label{prop:mc_accuracy}
Let $\{\mathbf{X}^{(k)}_{0:T}\}_{k=1}^K$ be $K$ independent samples of the observation path distribution, and let the corresponding smoother distributions $p_t(\mathbf y\mid \mathbf X_{0:T}^{(k)})$ be exactly known (such as for conditional Gaussian nonlinear systems). Consider the Monte Carlo estimator
\begin{equation}
\widehat I_K(t)
=
\frac{1}{K}\sum_{k=1}^K
G_t(\mathbf X^{(k)}_{0:T}).
\end{equation}

We can also write $\widehat{I_K}(t)$ as,

\begin{equation}
\widehat I_K(t)
=
\frac{1}{K}\sum_{k=1}^K
\int
\varphi\big(t,\mathbf X^{(k)}(t),\mathbf y\big)\,
p_t(\mathbf y\mid \mathbf X_{0:T}^{(k)})\,
\mathrm d\mathbf y,
\end{equation}
where $\varphi$ is any integrable and measurable test function and $X^{(k)}(t)$ denotes the realization of the observed trajectory at time t in the $k$-th sample. The corresponding empirical mixture density $\widehat{p_t}^K(\mathbf{x}, \mathbf{y})$ is analogously defined as

\begin{equation}\label{eq:mix_dens}
    \widehat{p_t}^K(\mathbf{x}, \mathbf{y}) = \frac{1}{K} \sum_{k=1}^K \delta(\mathbf{x} - \mathbf{X}^{(k)}) p_t(\mathbf y\mid \mathbf X_{0:T}^{(k)}).
\end{equation}

 Then $\widehat I_K(t)$ is an unbiased estimator of $\mathbb E_{p_t(\mathbf{x},\mathbf{y})}[\varphi(t,\mathbf X(t),\mathbf Y(t))]$, and
\begin{equation}
\mathrm{Var}_{p(X_{0:T})}(\widehat I_K(t))
=
\frac{1}{K}\,
\mathrm{Var}_{p(X_{0:T})}\big(G_t(\mathbf X_{0:T})\big).
\end{equation}
Consequently, the root-mean-square Monte Carlo error is of order $\mathcal O(K^{-1/2})$, arising from finite-sample approximating the path-wise law of the observed trajectories, where this sampling error does not depend explicitly on the hidden dimension of $\mathbf Y$, avoiding the latent curse of dimensionality.
\end{proposition}
\begin{proof}[Proof of Proposition \ref{prop:mc_accuracy}]
By definition and Proposition \ref{prop:smoothing_mixture},
\begin{equation}
G_t(\mathbf X^{(k)}_{0:T})
=
\mathbb E_{p_t(\mathbf y\mid \mathbf X_{0:T}^{(k)})}\big[\varphi(t,\mathbf X^{(k)}(t),\mathbf Y(t))\mid \mathbf X^{(k)}_{0:T}\big] =
\int
\varphi\big(t,\mathbf X^{(k)}(t),\mathbf y\big)\,
p_t(\mathbf y\mid \mathbf X_{0:T}^{(k)})\,
\mathrm d\mathbf y.
\end{equation}
 
Hence the Monte Carlo estimator can be written as
\begin{equation}
\widehat I_K(t)
=
\frac{1}{K}\sum_{k=1}^K G_t(\mathbf X^{(k)}_{0:T}).
\end{equation}
Since the trajectories $\mathbf X^{(k)}_{0:T}$ are independent and identically distributed, $\{G_t(\mathbf X^{(k)}_{0:T})\}_{k=1}^K$ are also i.i.d.\ random variables as measurable functionals of $\mathbf X_{0:T}$. Therefore, from Proposition \ref{prop:smoothing_mixture}
\begin{equation}
\mathbb E_{p(X_{0:T})}[\widehat I_K(t)]
=
\frac{1}{K}\sum_{k=1}^K \mathbb E_{p(X_{0:T})}[G_t(\mathbf X_{0:T}^{(k)})]
=
\mathbb E_{{p_t(\mathbf{x},\mathbf{y})}}[\varphi(t,\mathbf X(t),\mathbf Y(t))],
\end{equation}
which proves the unbiasedness. Moreover, we similarly have
\begin{equation}
\mathrm{Var}_{p(X_{0:T})}(\widehat I_K(t))
=
\mathrm{Var}_{p(X_{0:T})}\!\left(
\frac{1}{K}\sum_{k=1}^K G_t(\mathbf X^{(k)}_{0:T})
\right)
=
\frac{1}{K}\,\mathrm{Var}_{p(X_{0:T})}\big(G_t(\mathbf X_{0:T})\big).
\end{equation}
It follows that the root-mean-square Monte Carlo error is of order $\mathcal O(K^{-1/2})$.

Finally, given $p_t(\mathbf y\mid \mathbf X_{0:T}^{(k)})$, the estimator samples only the distribution of the observed trajectories $\mathbf X_{0:T}$. The integration over the hidden variable $\mathbf Y$ is performed analytically inside each conditional component. Therefore, the sampling error does not depend explicitly on the hidden dimension of $\mathbf Y$ when $p_t(\mathbf y\mid \mathbf X_{0:T}^{(k)})$ is known, such as the case of conditional Gaussian nonlinear systems.
\end{proof}
The key point is that given the posterior distribution $p_t(\mathbf y\mid \mathbf X_{0:T}^{(k)})$, $\widehat I_K(t)$ is an empirical average of the functional $G_t(\mathbf X_{0:T})$, which depends only on the observed trajectory. The hidden variable $\mathbf Y$ is integrated out analytically within each conditional component. Therefore, the Monte Carlo error is governed by the variability of $G_t(\mathbf X_{0:T})$ in the $\mathbf X$-path space, rather than by direct sampling in the hidden state space. This sampling structure can be particularly advantageous in multiscale systems with many unresolved or hidden degrees of freedom, like those arising in turbulent flows, geophysical fluid dynamics, and other complex multiscale applications \cite{majda2005information, chen2018conditional}. In the examples below, even when the full coupled system is simulated by Monte Carlo, this theorem explains why the mixture approximation does not introduce an additional direct sampling requirement in the hidden-state dimension.

Note, however, that the theorem requires an explicit posterior distribution, which is often unknown in general nonlinear systems. Thus, the general case has additional sampling demands on the hidden state through numerical data assimilation methods. This highlights the substantial computational advantage of the conditional Gaussian structure, which ensures a tractable posterior distribution while still capturing complex nonlinear dynamics.

This mixture perspective provides a statistical characterization of the hidden variables associated with extreme events. It will be complemented later by a path-wise analysis, where filtering and smoothing are compared to assess whether the mechanisms underlying extreme events can be reliably detected online or require retrospective reconstruction.

\subsection{Relative entropy and information-theoretic diagnostics}

To quantify differences between probability distributions arising in both the path-wise and statistical analyses below, we adopt the relative entropy \cite{kullback1951information, kullback1997information}, also known as Kullback-Leibler (KL) divergence. For two probability densities $p(\mathbf{y})$ and $q(\mathbf{y})$ defined on the same state space, the relative entropy of $p$ with respect to $q$ is
\begin{equation}\label{eq:KL_general}
\KL\big(p\,\|\,q\big)
=
\int p(\mathbf{y}) \log \frac{p(\mathbf{y})}{q(\mathbf{y})}\,\mathrm{d}\mathbf{y},
\end{equation}
where the integral is taken over the common support of $p$ and $q$.

The relative entropy is always nonnegative and vanishes if and only if $p=q$ almost everywhere. Additionally, it is invariant under general nonlinear changes of variables. It therefore provides a natural information-theoretic measure of discrepancy between two probabilistic descriptions of the hidden variables. Unlike a pointwise error metric, it compares entire distributions and thus captures differences in both coherent signals and uncertainty. Further, the logarithmic form is particularly well suited for studying extreme events as it transforms multiplicative contributions to additive, ensuring that differences in the tails contribute meaningfully rather than being dominated by high probability regions \cite{kleeman2011information}.  

In the present work, the relative entropy will be used in three different settings:

\begin{itemize}
\item \textbf{Filtering versus smoothing.}
Comparing
\(
p\big(\mathbf{Y}(t)\mid \mathbf{X}_{0:T}\big)
\)
with
\(
p\big(\mathbf{Y}(t)\mid \mathbf{X}_{0:t}\big)
\),
which quantifies the additional information about the hidden states gained from future observations and helps identify the onset of extreme-event mechanisms.

\item \textbf{Full smoothing versus finite-lag smoothing.}
Comparing the full smoother with a truncated smoother based only on observations up to $t+\tau$, which quantifies how much future information is needed to effectively reconstruct the hidden dynamics and thereby defines temporal influence ranges of potential triggering mechanisms for extreme events.

\item \textbf{Unconditional versus extreme-event-conditioned statistics.}
Comparing the marginal distribution
\(
p_t(\mathbf{y})
\)
with the conditional distribution
\(
p_t(\mathbf{y}\mid \mathcal E)
\),
where $\mathcal E$ denotes an extreme-event set in the observed variables as defined by Equation \eqref{eq:EE_set}, which reveals which hidden-state structures are preferentially associated with extreme events.
\end{itemize}

These three uses correspond, respectively, to (i) trajectory-wise onset detection, (ii) trajectory-wise influence quantification, and (iii) statistical characterization of hidden mechanisms.

A particularly important case arises when both distributions are Gaussian:
\begin{equation}
p(\mathbf{y}) = \mathcal N(\boldsymbol\mu_p,\mathbf R_p) \qquad\mbox{and}\qquad
q(\mathbf{y}) = \mathcal N(\boldsymbol\mu_q,\mathbf R_q),
\end{equation}
where $\mathbf{y}\in\mathbb{R}^{n_\mathbf{Y}}$. Then the relative entropy admits the closed-form formula
\begin{equation}\label{eq:KL_gaussian_general}
\KL\big(p\,\|\,q\big)
=
\frac{1}{2}
\left[(\boldsymbol\mu_q-\boldsymbol\mu_p)^\mathtt{T}
\mathbf R_q^{-1}
(\boldsymbol\mu_q-\boldsymbol\mu_p)\right] + \frac{1}{2}\left[\mathrm{tr}\!\left(\mathbf R_q^{-1}\mathbf R_p\right)
-
n_\mathbf{Y}
+
\log\frac{\det\mathbf R_q}{\det\mathbf R_p}
\right].
\end{equation}

The first term on the right-hand side of \eqref{eq:KL_gaussian_general} is the signal and measures the discrepancy in the means of the distributions, weighted by the covariance inverse $\mathbf{R}_q^{-1}$. The second term on the right-hand side is the dispersion, which quantifies the differences in the covariance structures, capturing discrepancies in the spread and orientation of the distributions. Since filtering, smoothing, and many reduced approximations in conditional Gaussian systems are Gaussian, \eqref{eq:KL_gaussian_general} provides an efficient diagnostic tool throughout the paper. In Sections~III and IV, this explicit formula will be used directly whenever the compared distributions are Gaussian or the non-Gaussian distributions are approximated by Gaussians.

\section{Mechanisms and Pathways of Extreme Events}\label{Sec:Mechanisms}
With the latent-state inference framework in place, formulated by Bayesian data assimilation, we now turn to the central goal of this paper: extracting the hidden mechanisms and pathways that generate observed extreme events.

We develop methods for diagnosing the hidden mechanisms that generate observed extreme events. The analysis proceeds along two complementary directions. First, we introduce trajectory-wise diagnostics that operate on individual realizations and identify when hidden precursors emerge, whether they are detectable in real time, and how long their influence persists. Second, we develop statistical diagnostics that aggregate information across many events in order to reveal common hidden-state structures, sensitive triggering directions, and representative latent pathways associated with extreme episodes. To explore potentially distinct triggering mechanisms, we also incorporate a classification step that separates extreme events into candidate pathway classes for subsequent dynamical interpretation.

The diagnostics in this section are formulated in terms of filtering and smoothing distributions and therefore apply broadly beyond conditional Gaussian models. When closed-form posteriors are unavailable, these quantities can be replaced by ensemble or particle approximations \cite{evensen2009data, asch2016data, johansen2008tutorial, alexander2005accelerated, taghvaei2023survey}.

\subsection{Path-wise study of extreme events}

We start by developing path-wise tools for identifying the onset and temporal influence of the hidden mechanisms associated with extreme events. The key idea is to compare different conditional distributions of the hidden variables along an individual realization of the observed state and to quantify the discrepancy between these latent representatives through the relative entropy. Since the hidden variables are not directly observed, this provides a practical way to infer when hidden precursors emerge, whether they are detectable in real time, and how long their influence persists before the observed extreme event occurs.

Throughout this subsection, let $t_\ast$ denote the time at which an observed extreme event occurs, identified through a prescribed criterion on $\mathbf{X}(t)$ such as the attainment of a local maximum.

\subsubsection{Filtering versus smoothing and hidden onset detection}

Recall that the filtering distribution $p_{\text{f}}(\mathbf y,t)=p\big(\mathbf Y(t)\mid \mathbf X_{0:t}\big)$ represents the best online estimate of the hidden state at time $t$, using only information available up to that time. In contrast, the smoothing distribution
$p_{\text{s}}(\mathbf y,t)
=
p\big(\mathbf Y(t)\mid \mathbf X_{0:T}\big)$ with $T\ge t_\ast$ uses future observations and therefore provides an improved retrospective reconstruction of the hidden-state evolution.

We quantify the discrepancy between these two distributions by the time-dependent relative entropy
\begin{equation}\label{eq:KLt_path}
{\KL}(t)
=
\KL\!\big(p_{\text{s}}(\cdot,t)\,\|\,p_{\text{f}}(\cdot,t)\big).
\end{equation}

The quantity $\KL(t)$ measures the information gain obtained by incorporating future observations \cite{andreou2026assimilative}. If $\KL(t)$ remains small, then the filtering estimate already captures the hidden dynamics reliably in real time. Large values of $\KL(t)$ indicate that essential precursor information is absent from the online estimate and only becomes visible retrospectively through smoothing.

To identify the onset of the hidden mechanisms leading to the extreme event, we examine $\KL(t)$ backward in time over a prescribed precursor window preceding the event peak at $t=t_\ast$. Let $[t_\ast-T_{\mathrm{pre}},\,t_\ast]$ denote the search interval, where $t_\ast$ is the extreme event's peak time and $T_{\mathrm{pre}}>0$ is a chosen look-back horizon. The onset time is defined as the first time at which the discrepancy between smoothing and filtering exceeds a prescribed threshold $\kappa$, namely
\begin{equation}\label{t_on}
t_{\mathrm{on}}(\kappa, T_{pre})
=
\inf\Big\{
t\in[t_\ast-T_{\mathrm{pre}},\,t_\ast]:
\KL(t)\ge \kappa
\Big\}.
\end{equation}
The rationale behind \eqref{t_on} is that once future observations become sufficiently informative about the hidden dynamics, the smoother and filter begin to differ substantially; the first persistent exceedance therefore marks the emergence of the latent precursor mechanisms that drive the forthcoming extreme event. 

This definition is natural because, far before the event, the hidden dynamics associated with the future extreme episode have typically not yet developed, so future observations provide little additional information at the early stages and $\KL(t)$ remains small. Near the event time $t_\ast$, part of the precursor information has already propagated into the observed variables, making the filtering estimate progressively closer to the smoothing estimate and reducing the discrepancy again. Therefore, the maximizer of $\KL(t)$ is expected to occur during the intermediate stage when the hidden triggering mechanism is active but has not yet become visible in the observed components, with the first exceedance of the threshold $\kappa$, $t_{\mathrm{on}}(\kappa, T_{pre})$, marking the onset of this latent mechanism.

From a dynamical perspective, $t_{\mathrm{on}}(\kappa, T_{pre})$ identifies the transition from a background state to a precursor state in which latent instabilities, unresolved forcing, or hidden nonlinear interactions begin to organize and induce the forthcoming extreme event. It thus provides a practical estimate of when the extreme event starts in the hidden variables, rather than when it becomes apparent in the observations alone. In this sense, the proposed diagnostic can be interpreted as a backward, path-wise information-based measure that identifies when hidden states begin to exert a causal influence on the observed extreme event \cite{andreou2025bridging}.

Note that if no threshold crossing occurs within the search window, one may set $t_{\mathrm{on}}(\kappa, T_{pre})=t_\ast$ or regard the event as having no clearly detectable precursor under the chosen criterion. The value of $t_{\mathrm{on}}(\kappa, T_{pre})$ depends on $\kappa$ and $T_{pre}$, which must be chosen adequately to balance the sensitivity to precursor dynamics with robustness to noise. For the dynamical systems explored in this work, case studies have shown that the inferred precursor onset is qualitatively robust to modest changes in \(\kappa\) and \(T_{\mathrm{pre}}\), with only moderate shifts in the estimated onset time that are within expectations.

\subsubsection{Full smoothing versus finite-lag smoothing and temporal influence range}

The previous diagnostic identifies when the hidden mechanisms begin. We now quantify how long their influence persists. This addresses a complementary question: whether the hidden precursor merely triggers the extreme event or continues to shape the system evolution beyond the event itself.

For each $t<t_\ast$ and lag $\tau>0$, define the finite-lag smoother \cite{cohn1994fixed, todling1998suboptimal}
\begin{equation}
p_{\text{s}}^{(\tau)}(\mathbf y,t)
=
p\big(\mathbf Y(t)\mid \mathbf X_{0:t+\tau}\big),
\end{equation}
which uses observations only up to the future time $t+\tau$ rather than the full interval $[0,T]$.

We compare the finite-lag smoother with the full smoother through
\begin{equation}\label{eq:KL_finite_tau}
\KL^{(\tau)}(t)
=
\KL\!\big(p_{\text{s}}(\cdot,t)\,\|\,p_{\text{s}}^{(\tau)}(\cdot,t)\big).
\end{equation}
When $\tau=0$, the finite-lag smoother reduces to the filtering distribution, and therefore $\KL^{(0)}(t)$ recovers the filter-smoother discrepancy $\KL(t)$ in \eqref{eq:KLt_path}.
For fixed $t$, the lagged KL divergence $\KL^{(\tau)}(t)$ is nonincreasing in expectation over the observable trajectories as $\tau$ increases, reflecting the incorporation of additional future observations \cite{andreou2026assimilative}. For an individual realization, however, this quantity may fluctuate with $\tau$ and need not decrease monotonically \cite{andreou2025bridging}. The prescribed tolerance $\eta>0$ provides a threshold-based definition of the influence horizon. Thus, we define the {thresholded influence range} at time $t$ by
\begin{equation}
\tau^{(\eta)}(t)
=
\inf\big\{\tau>0:\KL^{(\tau)}(t)\le\eta\big\}.
\end{equation}
This is the minimal amount of future information required to accurately reconstruct the hidden state at time $t$ under an $\eta$-tolerance. If $\tau^{(\eta)}(t)$ is large, then the hidden mechanisms active at time $t$ continue to influence the system over a long future horizon.

Since $\tau^{(\eta)}(t)$ depends on the subjective tolerance level $\eta$, it is also useful to introduce a threshold-free measure. We therefore define the {integrated influence range} \cite{andreou2026assimilative}
\begin{equation}\label{T_range}
\mathcal T(t)
=
\frac{1}{\KL(t)}
\int_{0}^{\KL(t)} \tau^{(\eta)}\,\mathrm d\eta.
\end{equation}
By uniformly averaging the threshold-dependent influence ranges over all admissible tolerance levels, this definition removes dependence on any single subjective choice of \(\eta\) while retaining the characteristic temporal extent of the hidden-state influence. The normalization by $\KL(t)$ ensures that $\mathcal T(t)$ has units of time.

The quantities $\tau^{(\eta)}(t)$ and $\mathcal T(t)$ are analogous to the autocorrelation function and decorrelation time, respectively, used to quantify memory in stochastic systems \cite{andreou2026assimilative}. The influence ranges measure the forward influence of hidden precursor dynamics on the development of an extreme event.
For each fixed $t$, the integrated influence range $\mathcal T(t)$ provides a trajectory-level measure of the typical time horizon over which hidden precursors affect the observed evolution. In particular, if $\mathcal T(t)$ extends beyond $t_\ast$, then the hidden mechanisms influence not only the onset of the event but also its subsequent evolution or recovery phase.

To summarize, the onset time in \eqref{t_on} and the influence ranges in \eqref{T_range} characterize distinct but complementary temporal aspects of hidden mechanisms. The former identifies when the precursor first emerges, while the latter quantifies how long its effects persist.
\subsection{Statistical study of extreme events}

The path-wise analysis in Sec.~III~A characterizes when hidden mechanisms emerge and how long their influence persists along individual trajectories. We now develop a complementary statistical viewpoint for studying extreme events.

\subsubsection{Information-theoretic sensitive directions for extreme-event triggering}

We begin by identifying which directions in the hidden-state space are most strongly associated with extreme events. Let $\mathcal E_t$ denote an extreme-event set defined through the observed variables at time $t$, for example,
\begin{equation}\label{eq:EE_set}
    \mathcal E_t=\{\mathbf X(t)\in\mathcal A\},
\end{equation}
where $\mathcal A$ corresponds to a threshold exceedance or another extreme-event criterion (e.g. rare transitions, excursions, or regime shifts) on $\mathbf{X}$ or a function of it (observable) \cite{coles2001introduction, farazmand2019extreme}. Using the conditional Gaussian mixture representation described in Section \ref{Sec:Framework}, one can construct both the unconditional hidden-state distribution
\begin{equation}\label{unconditional_distribution}
p_t(\mathbf y)=p(\mathbf Y(t) = \mathbf y),
\end{equation}
and the extreme-event-conditioned hidden-state distribution
\begin{equation}\label{conditional_distribution}
p_t(\mathbf y\mid\mathcal E_t)=p(\mathbf Y(t) = \mathbf y\mid\mathcal E_t).
\end{equation}
Comparing these two distributions reveals how hidden states are reorganized during the onset and active phases of extreme events.

A fundamental question is to determine which direction in the hidden-state space is most sensitive to the occurrence of observed extremes. In other words, we seek the direction along which the hidden variables exhibit the strongest statistical change when conditioning on an extreme event. To this end, define the scalar projection
\begin{equation} \label{eq:scalar_proj}
z_{\mathbf v}(t)=\mathbf v^\mathtt{T}\mathbf Y(t),
\end{equation}
where $\mathbf v\in\mathbb R^{n_\mathbf{Y}}$ is a unit vector. Let $p_t^{\mathbf v}(z)$ and $p_t^{\mathbf v}(z\mid\mathcal E_t)$ denote the corresponding projected distributions under the unconditional and event-conditioned laws.

We define the directional sensitivity score by
\begin{equation}
\mathcal J_t(\mathbf v)
=
\KL\!\left(
p_t^{\mathbf v}(z\mid\mathcal E_t)
\,\|\,
p_t^{\mathbf v}(z)
\right),
\qquad \|\mathbf v\|=1.
\end{equation}
The most sensitive direction along which the hidden variables exhibit the strongest statistical change in the event-conditioned distribution is then given by
\begin{equation}\label{eq:vstar_general}
\mathbf v_\ast(t)
=
\arg\max_{\|\mathbf v\|=1}\mathcal J_t(\mathbf v).
\end{equation}

This direction is well-defined, exists, and is unique under mild conditions \cite{boyd2004convex,nesterov2018lectures}. Equation \eqref{eq:vstar_general} identifies, within the family of one-dimensional linear projections defined in \eqref{eq:scalar_proj}, the direction along which the hidden-state distribution is altered most strongly by conditioning on the extreme event. It therefore provides a statistical characterization of the hidden pattern most relevant for triggering or sustaining extremes. The same information-theoretic construction extends naturally to multiple sensitive directions by replacing the vector $\mathbf v$ with a low-dimensional projection matrix and optimizing over the corresponding subspace. More general nonlinear projections can likewise be treated by optimizing the projected divergence over an appropriate parameterized family of mappings. Here, we focus on the one-dimensional linear case because it provides a direct and interpretable diagnostic for the applications considered below.

In general, the projections of both distribution \eqref{unconditional_distribution} and \eqref{conditional_distribution} are non-Gaussian, even in conditional Gaussian nonlinear systems. Nevertheless, Gaussian approximations with matched moments provide a tractable surrogate that facilitates the optimization problem and yields an approximate sensitive direction.

\begin{proposition}[Most sensitive directions under a Gaussian approximation]
Suppose the two distributions are approximated by Gaussian laws with matching moments:
\begin{align}
p_t(\mathbf y)
&\approx
\mathcal N(\boldsymbol\mu_0(t),\mathbf R_0(t)),\\
p_t(\mathbf y\mid\mathcal E_t)
&\approx
\mathcal N(\boldsymbol\mu_{\mathcal E}(t),\mathbf R_{\mathcal E}(t)).
\end{align}
Then the projected distributions of \eqref{eq:scalar_proj} are one-dimensional Gaussians:
\begin{align}
p_t(z_{\mathbf v})
&=
\mathcal N\!\big(\mathbf v^\mathtt{T}\boldsymbol\mu_0,\mathbf v^\mathtt{T}\mathbf R_0\mathbf v\big),\\
p_t(z_{\mathbf v}\mid\mathcal E_t)
&=
\mathcal N\!\big(\mathbf v^\mathtt{T}\boldsymbol\mu_{\mathcal E},\mathbf v^\mathtt{T}\mathbf R_{\mathcal E}\mathbf v\big).
\end{align}
Using the closed-form expression of the relative entropy in \eqref{eq:KL_gaussian_general}, we obtain
\begin{equation}\label{eq:projected_KL}
\mathcal J_t(\mathbf v)
=
\frac12\left[
\frac{\mathbf v^\mathtt{T}\mathbf R_{\mathcal E}\mathbf v}{\mathbf v^\mathtt{T}\mathbf R_0\mathbf v}
+
\frac{\big(\mathbf v^\mathtt{T}(\boldsymbol\mu_{\mathcal E}-\boldsymbol\mu_0)\big)^2}
{\mathbf v^\mathtt{T}\mathbf R_0\mathbf v}
-1
+
\log
\frac{\mathbf v^\mathtt{T}\mathbf R_0\mathbf v}
{\mathbf v^\mathtt{T}\mathbf R_{\mathcal E}\mathbf v}
\right].
\end{equation}
This expression shows that directional sensitivity contains two distinct contributions: shifts in the mean state and changes in variability or uncertainty.
\end{proposition}

The Gaussian approximation in \eqref{eq:projected_KL} also clarifies whether the discrepancy between the conditional and unconditional distributions is primarily driven by mean shifts or by covariance changes.

\begin{corollary}[Mean-shift dominated regime]
Suppose the covariance change is negligible, so that $\mathbf R_{\mathcal E}(t)\approx \mathbf R_0(t)$. Then \eqref{eq:projected_KL} reduces to maximizing
\begin{equation}
\frac{\big(\mathbf v^\mathtt{T}(\boldsymbol\mu_{\mathcal E}-\boldsymbol\mu_0)\big)^2}
{\mathbf v^\mathtt{T}\mathbf R_0\mathbf v}.
\end{equation}
The optimizer satisfies
\begin{equation}\label{eq:vstar_mean}
\mathbf v_\ast(t)
\propto
\mathbf R_0^{-1}(t)\big(\boldsymbol\mu_{\mathcal E}(t)-\boldsymbol\mu_0(t)\big).
\end{equation}
Thus, the most sensitive direction is the covariance-weighted mean shift between the event-conditioned and background hidden states.
\end{corollary}

Equation \eqref{eq:vstar_mean} shows that in the mean-shift dominated regime, the most sensitive direction emphasizes coherent changes in the mean that are large relative to the background variability encoded in $\mathbf R_0^{-1}$.

\begin{corollary}[Covariance-change dominated regime]
Suppose instead that the mean shift is negligible, so that $\boldsymbol\mu_{\mathcal E}(t)\approx\boldsymbol\mu_0(t)$. Then sensitivity in the hidden state under the extreme event is driven entirely by changes in uncertainty. A natural characterization is obtained through the generalized covariance contrast
\begin{equation}
\mathbf M(t)
=
\mathbf R_0^{-1/2}(t)\mathbf R_{\mathcal E}(t)\mathbf R_0^{-1/2}(t).
\end{equation}
The most sensitive covariance-only directions are then given by the eigenvectors of $\mathbf M(t)$, corresponding to the generalized eigenproblem $\mathbf R_{\mathcal E}(t)\mathbf{v}=\lambda \mathbf R_{0}(t)\mathbf{v}$, associated with whichever of the smallest and largest generalized eigenvalues yields the larger value of $\lambda-1-\log\lambda$. These correspond to the directions in which variability is most amplified or suppressed during extreme events.
\end{corollary}

The matrix $\mathbf M(t)$ therefore plays the role of a relative covariance operator, identifying directions of maximal variability change between the background and event-conditioned hidden states.

The study of sensitive directions can be carried out at different times depending on the objective. At the event time $t=t_\ast$, the vector $\mathbf v_\ast(t_\ast)$ identifies hidden structures most strongly associated with the realized extreme state. At the onset time $t=t_{\mathrm{on}}(\kappa, T_{pre})$, the vector $\mathbf v_\ast(t_{\mathrm{on}}(\kappa, T_{pre}))$ identifies hidden directions most relevant for triggering the event. Over an entire time interval, the trajectory $t\mapsto \mathbf v_\ast(t)$ reveals transitions from precursor mechanisms to growth dynamics and finally to recovery processes.

The direction $\mathbf v_\ast(t)$ therefore provides a compact summary of how hidden variables contribute to extreme events. It can help identify the most influential latent components, informing risk assessment and preparation for future extreme events.

As a final remark, the hidden-state distributions used to compute sensitive directions can be constructed from either filtering-based or smoothing-based conditional Gaussian mixtures as defined in Section \ref{Sec:Framework}. The smoothing version is most suitable for mechanism discovery because it uses the full observation record and provides the most accurate reconstruction of hidden dynamics. The filtering version \cite{chen2018efficient}, by contrast, uses only real-time information and therefore quantifies online detectability. Comparing the sensitive directions obtained from these two constructions offers an additional diagnostic. If the filtering and smoothing directions are already close before the event, then the hidden precursor is detectable in real time. If they differ substantially, then the triggering mechanism remains concealed until future observations are incorporated retrospectively. 

\subsubsection{Most probable hidden paths and representative trajectories}\label{Subsec:most_probable}

The distribution-based analysis in Sec.~III~B~1 identifies latent directions and changes in the hidden-state distributions associated with extreme events. A complementary question concerns the temporal evolution of the hidden variables: among all hidden trajectories consistent with the observations, what is the dominant route toward an extreme event? The smoothing distribution provides a natural probabilistic framework for addressing this question.

For simplicity, we assume a discrete-time setting with observations at times $t_0<t_1<\ldots < t_N$ and consider an observed realization $\mathbf X_{0:N}$ containing an extreme event. Let
\begin{equation}
p(\mathbf Y_{0:N}\mid \mathbf X_{0:N})
\end{equation}
denote the posterior law of the hidden path
\[
\mathbf Y_{0:N}=(\mathbf Y(t_0),\mathbf Y(t_1),\dots,\mathbf Y(t_N))
\]
conditioned on the full observation record. A most probable hidden trajectory is naturally defined through a maximum a posteriori criterion. In the conditional Gaussian setting, this construction becomes explicit and computationally tractable \cite{liptser2001statistics}.

\begin{proposition}[Event-wise most probable hidden path]\label{prop:event_path}
Suppose the smoothing posterior of the hidden path is Gaussian:
\begin{equation}\label{eq:smoother_path_gaussian_prop}
\mathbf Y_{0:N}\mid \mathbf X_{0:N}
\sim
\mathcal N\!\big(\boldsymbol\mu^{\text{\normalfont{s}}}_{0:N},\boldsymbol\Sigma^{\text{\normalfont{s}}}_{0:N}\big).
\end{equation}
This is true for conditional Gaussian nonlinear systems \cite{liptser2001statistics}. Define the most probable hidden path by
\begin{equation}\label{eq:MAP_path_prop}
\mathbf Y^\ast_{0:N}
=
\arg\max_{\mathbf Y_{0:N}}\{
p(\mathbf Y_{0:N}\mid \mathbf X_{0:N})\}.
\end{equation}
Then, due to the Gaussianity of the posterior, the most probable path coincides with the posterior mean path:
\begin{equation}\label{eq:MAP_equals_mean_prop}
\mathbf Y^\ast_{0:N}
=
\boldsymbol\mu^{\text{\normalfont{s}}}_{0:N}.
\end{equation}
Moreover, deviations from the dominant path are quantified by the quadratic score
\begin{equation}\label{eq:quadratic_score_prop}
\mathcal S(\mathbf Y_{0:N})
=
\frac12
(\mathbf Y_{0:N}-\boldsymbol\mu^{\text{\normalfont{s}}}_{0:N})^\mathtt{T}
(\boldsymbol\Sigma^{\text{\normalfont{s}}}_{0:N})^{-1}
(\mathbf Y_{0:N}-\boldsymbol\mu^{\text{\normalfont{s}}}_{0:N}),
\end{equation}
which is proportional to the negative log-posterior up to an additive constant.
\end{proposition}
The quadratic score in \eqref{eq:quadratic_score_prop} measures the Mahalanobis distance of a candidate hidden trajectory from the posterior mean path, i.e. the most probable path, under the smoothing distribution. Since the posterior density is proportional to $\exp(-\mathcal S)$, trajectories with smaller scores are more probable given the observations. This score can therefore be used to rank alternative hidden reconstructions, construct credible tubes around the dominant path, and assess how sharply the hidden mechanisms are constrained by the data.

Importantly, this proposition justifies the use of the smoothed mean trajectory as the dominant hidden evolution associated with an individual extreme event, while the smoothing covariance quantifies the uncertainty around that path.

When many extreme events are available, one seeks a representative trajectory for a family of events sharing similar dynamical signatures. A direct arithmetic average of event-wise dominant paths may be misleading when some hidden trajectories are poorly constrained by the data. It is therefore natural to assign greater weight to events whose smoothing posteriors are more concentrated or less uncertain.

Let $\hat{\mathbf{Y}}^{(k)}_{0:\tilde{N}}$ denote the aligned dominant or most probable path of event $k$ expressed in a relative time coordinate with discrete times denoted $\tau_0 < \tau_1 < \ldots < \tau_{\tilde{N}}$ (for example, aligned and measured from the onset time $t_{on}(\kappa, T_{pre})$ in \eqref{t_on} or peak time $t_*$), and let $\boldsymbol\Sigma^{\text{s},(k)}_{0:\tilde{N}}$ be the corresponding path covariance.

\begin{proposition}[Likelihood-weighted representative path]\label{prop:stat_path}
Consider a collection of $K$ extreme events with aligned dominant paths $\{\hat{\mathbf Y}^{(k)}_{0:\tilde{N}}\}_{k=1}^K$. Let
\begin{equation}\label{eq:weight_score_prop}
J_k
=
\frac12\log\det\boldsymbol\Sigma^{\text{\normalfont{s}},(k)}_{0:\tilde{N}},
\end{equation}
or, more generally, any other scalar measure of posterior path uncertainty. Define the normalized weights
\begin{equation}\label{eq:weights_prop}
w_k
=
\frac{\exp(-J_k)}{\sum_{j=1}^K\exp(-J_j)}.
\end{equation}
Then the likelihood-weighted representative path is
\begin{equation}\label{eq:weighted_path_prop}
\bar{\mathbf Y}_{0:\tilde{N}}
=
\sum_{k=1}^K w_k\,\hat{\mathbf Y}^{(k)}_{0:\tilde{N}}.
\end{equation}
This construction defines a probability distribution over the negative uncertainty scores of the aligned dominant paths that assigns larger weights to events with smaller posterior uncertainty, which therefore contribute more strongly to the representative hidden trajectory. This representative path is a statistical summary of the aligned event-wise most probable trajectories and need not itself correspond to an exact realization of the dynamical system.
\end{proposition}

\noindent
The weighted path \eqref{eq:weighted_path_prop} summarizes the dominant latent evolution of a class of extremes while reducing the influence of poorly identified hidden reconstructions.

To understand the statistical behavior of the hidden dynamics relative to the observed dynamics, a representative observed path can similarly be constructed using the same event weights $w_k$ in \eqref{eq:weights_prop} applied to the aligned observed extreme paths. This ensures that the representative observed and hidden trajectories summarize the same weighted ensemble of events. Such a representative path depends on the chosen threshold used to define the extreme-event ensemble. Nevertheless, under typical regularity and smoothness conditions satisfied in practice, modest variations of this threshold are expected to lead to qualitatively similar representative observed and hidden pathways.

Beyond estimating a representative trajectory, it is equally important to quantify its uncertainty. This uncertainty has two distinct sources: posterior uncertainty within each event and variability across different events.

\begin{proposition}[Uncertainty decomposition for representative paths]\label{prop:uq_path}
Let $\bar{\mathbf Y}_{0:\tilde{N}}$ be the representative path defined in \eqref{eq:weighted_path_prop}. Denote by $\boldsymbol\Sigma^{\text{\normalfont{s}},(k)}_{0:\tilde{N}}$ the marginal smoother covariance of event $k$ at relative times $\tau_0<\tau_1<\ldots<\tau_{\tilde{N}}$. Then the total uncertainty around the representative path is \cite{kurkoski2009single, minka2013expectation}
\begin{equation}\label{eq:total_uncertainty_prop}
\bar{\boldsymbol\Sigma}_{0:\tilde{N}}
=
\sum_{k=1}^K w_k
\left[
\boldsymbol\Sigma^{\text{\normalfont{s}},(k)}_{0:\tilde{N}}
+
\big(\hat{\mathbf Y}^{(k)}_{0:\tilde{N}}-\bar{\mathbf Y}_{0:\tilde{N}}\big)
\big(\hat{\mathbf Y}^{(k)}_{0:\tilde{N}}-\bar{\mathbf Y}_{0:\tilde{N}}\big)^\mathtt{T}
\right].
\end{equation}
The first term represents the within-event posterior uncertainty, while the second term captures the between-event variability of the dominant paths. This is the standard within- and between-component covariance decomposition, corresponding to the law of total covariance for the weighted mixture induced by the weights \(w_k\).
\end{proposition}

\noindent
This decomposition distinguishes uncertainty caused by imperfect hidden-state reconstruction from genuine heterogeneity among events. Hence it provides a confidence envelope around the representative trajectory together with a statistical measure of mechanism variability. Together, the representative path and total uncertainty provide a tractable Gaussian characterization of the dominant hidden-event dynamics and their variability across events.

\subsection{Classification of heterogeneous extreme-event pathways}

The trajectory-wise and statistical analyses provide complementary rather than redundant information about hidden extreme-event mechanisms. The trajectory-wise viewpoint conditions on individual observed realizations and is therefore suited to resolving event-specific onset times, persistence of hidden-state influence, and online versus retrospective detectability. In contrast, the statistical viewpoint aggregates information across an ensemble of extreme events and identifies systematic changes in the hidden-state distribution, including both shift in the mean state and changes in its uncertainty. The two viewpoints are expected to give consistent mechanistic interpretations when the extreme-event population is relatively homogeneous and the dominant precursor appears as a coherent hidden-state signal shared across events. They may differ when extreme events are heterogeneous, when precursor timing varies substantially between realizations, or when the dominant statistical signature appears primarily through changes in covariance rather than through a representative mean trajectory. In such cases, the two diagnostics should be interpreted as resolving different aspects of the same underlying dynamics rather than as competing descriptions.

In complex systems, treating all extremes as a single population may nevertheless obscure important structures and mix qualitatively different pathways. Here, we use the term pathway to describe the dynamical route through state space by which an extreme event develops, whereas a mechanism refers to the underlying dynamical process responsible for generating or amplifying that evolution. Depending on the selected feature representation, a pathway may be characterized through observable variables, reconstructed hidden variables, or their joint evolution. We therefore introduce a classification step that separates extreme events into candidate pathway classes, after which the trajectory-wise and statistical diagnostics can be applied within each group.

Suppose a collection of extreme-event realizations has been identified through an event criterion $\mathcal E_t$ in \eqref{eq:EE_set}. For each event $k$, we construct a feature vector $\mathbf z^{(k)} \in \mathbb R^m$ that summarizes relevant aspects of the event evolution both through the observables and the reconstructed hidden variables.

A key point is that these features are not restricted to static snapshots of the system state. Instead, they may include dynamic and mechanistic information extracted from the full event evolution. Typical examples include onset times, precursor durations, growth rates, peaking times, recovery times, or lead-lag relationships; for concrete examples of such features, see Appendixes A and B. When available, it is advantageous to combine information from both the observed variables and reconstructed hidden dynamics: observable features characterize differences in event morphology, while reconstructed hidden-state and dynamical features provide complementary information about the underlying evolution. Depending on which features are included, the resulting pathway classes may therefore reflect structure in the observable variables, the hidden variables, or their joint dynamics. Thus, each extreme event is embedded into a feature space that reflects not only its instantaneous state, but also the dynamical pathway through which the event develops.

To organize extreme events into pathway classes, we apply a clustering algorithm to the set of chosen feature vectors $\{\mathbf z^{(k)}\}_{k=1}^K$. Depending on the data structure, one may use methods such as $k$-means, Gaussian mixture models, hierarchical clustering, or density-based clustering. This yields a cluster partition $\{\mathcal C_1,\ldots,\mathcal C_M\}$, where each cluster contains events with similar signatures in the selected feature representation.

Importantly, statistical separation in feature space does not by itself imply separation of the underlying physical mechanisms. A single mechanism may generate multiple dynamically distinct pathways because of the system's stochasticity, initial-condition dependence, or nonlinear evolution, while distinct mechanisms may project similarly onto a restricted set of features. We therefore interpret the clusters as candidate pathway classes whose physical meaning is established through the subsequent trajectory-wise and distribution-based dynamical diagnostics.

\textbf{Remark III.7: }(Cluster-based pathway separation)
\textit{If extreme events form robust and distinct groups in a dynamically informed feature space, these groups define candidate extreme-event pathway classes. Such pathway classes may correspond to distinct physical mechanisms, but this correspondence requires additional dynamical interpretation and does not follow from cluster separation alone.}

The practical value of this classification is that the trajectory-wise and statistical analyses can subsequently be applied within each candidate pathway class. This enables class-specific identification of precursors, sensitive directions, representative paths, and probability laws that may otherwise be masked in the full mixed population. The resulting dynamical diagnostics can then be used to determine whether different pathway classes arise from distinct mechanisms, different realizations of a common mechanism, or mixed interactions among several mechanisms. For example, one cluster may correspond to events driven primarily by mean shifts in hidden variables, while another may reflect variance amplification, delayed precursor growth, or intermittent bursts. Thus, clustering organizes heterogeneous observed extremes into candidate pathway classes for subsequent dynamical interpretation, with the resolved pathways reflecting the observable, reconstructed hidden-state, or joint features included in the chosen representation.

As with any unsupervised classification, the resulting partition may depend on the selected feature representation, preprocessing, and number of clusters. Its robustness should therefore be assessed within each application, for example by varying the feature set, removing or adding individual features, modifying the preprocessing, and examining whether the resulting event assignments and dynamical interpretations remain stable. The number of clusters may likewise be guided by structure in the feature space, dynamical interpretability, or prior physical knowledge, and should not be generally regarded as uniquely determined. The purpose of the clustering is therefore not to provide a universal mechanism classification, but to identify candidate pathway classes whose dynamical significance can subsequently be evaluated using the trajectory-wise and distribution-based diagnostics.

\section{Numerical Investigations of Extreme Events in Prototype Systems}\label{Sec:Numerics}

We now illustrate the proposed methodology through three prototype systems of increasing complexity. The examples are chosen to emphasize different aspects of the framework rather than to exhaust all possible applications. The first example is a minimal intermittent model with one hidden variable and is used to study path-wise filtering versus smoothing hidden-state reconstruction and online detectability of extreme events. The second example introduces multiple hidden mechanisms for triggering extremes (stochastic damping- and force-driven processes) and is used to study clustering, representative trajectories, and mechanism-dependent conditional statistics. The first two examples use conditional Gaussian models to illustrate the diagnostics in settings where exact posterior inference is available. The third example considers a topographic flow model with richer spatial structure and demonstrates how the ideas developed here extend from the analytically tractable conditional Gaussian structure to more general nonlinear systems via Monte Carlo simulation.

\subsection{Hidden precursors and the onset of extreme events in an intermittent mode}

We begin with a two-dimensional stochastic system consisting of one observed variable $u(t)$ and one hidden variable $\gamma(t)$:
\begin{equation}\label{SPEKF_M_num}
\begin{split}
\frac{\d u}{\d t} &= -d_u u +c\gamma u + F_u + \sigma_u \dot{W}_u,\\
\frac{\d \gamma}{\d t} &= -d_\gamma \gamma - cu^2 + F_\gamma + \sigma_\gamma \dot{W}_\gamma.
\end{split}
\end{equation}
Here, $u$ denotes the observed mode, while $\gamma$ is a stochastic damping coefficient. The variable $\gamma$ is hidden and modulates the instantaneous growth or decay of the observed mode through the stochastic damping term $c\gamma u$ with $c>0$. Extreme events in $u$ occur when the effective damping becomes sufficiently weak, or equivalently, when $\gamma$ becomes large and positive, $\gamma > \frac{d_u}{c}$. This system can be interpreted as a reduced-order stochastic parameterization of a single Fourier mode within a larger intermittent system. It has been widely used as a prototype model for studying forecasting, data assimilation, and uncertainty quantification of extreme events in atmospheric low-frequency variability \cite{majda2012lessons, branicki2013non}. It has also been adopted as a cheap and analytically tractable surrogate model for stochastic superparametrization \cite{keating2012new} and for filtering the Navier-Stokes equation and other complex geophysical systems \cite{branicki2018accuracy, harlim2010filtering}.

A nonlinear feedback term $-cu^2$ is included in the governing equation for the stochastic damping $\gamma$. This ensures that the total energy associated with the nonlinear terms ($c\gamma u$ and $-cu^2$) is conserved, a key feature for mimicking realistic turbulent dynamics \cite{majda2013physics}. To see this, multiply the two equations by $u$ and $\gamma$, respectively, and retain only the nonlinear terms on the right-hand side. Then the left-hand side becomes the time derivative of the total energy $E= \frac{1}{2}(u^2+\gamma^2)$, while the right-hand side vanishes. Dynamically, once an extreme event in $u$ is triggered by the stochastic damping, the feedback term $-cu^2$ reduces the amplitude of $\gamma$, after which $\gamma$ drives $u$ back toward its normal state. This mechanism prevents non-physical amplification and yields more realistic intermittent bursts. Because the model contains only one latent variable, it provides the simplest setting in which the hidden triggering mechanism can be visualized directly.

The parameter values used in \eqref{SPEKF_M_num} are as follows:
\begin{equation}
F_u = 1,\quad
F_\gamma = 1,\quad
d_\gamma = 0.8,\quad
d_u = 0.8,\quad
c = 1.2,\quad
\sigma_\gamma = 2,\quad
\sigma_u = 0.5,
\end{equation}
with numerical integration time step $\Delta t=0.005$.

Panel (a) of Figure \ref{Fig:SPEKF_M_Paths} displays the observed trajectory of $u$, with extreme events highlighted in red. For this case study, an extreme event is defined as an excursion whose peak amplitude exceeds the 90th percentile of all event amplitudes. Because the deterministic forcing satisfies $F_u>0$, the most prominent excursions occur on the positive side. These events are associated with intermittent episodes of positive effective damping (anti-damping), generated when the stochastic damping variable $\gamma$ rises above the threshold $d_u/c$ (Panel (b)), so that the net damping changes sign and transient growth becomes possible. Panel (c) shows the representative observed trajectory of $u$ aligned at the event peak, constructed using the uncertainty-based event weights introduced in Section \ref{Subsec:most_probable} (by supplying the $u$-equation in \eqref{SPEKF_M_num} to \eqref{eq:weighted_path_prop}). Panels (d) and (e) present the corresponding representative hidden trajectories inferred from the smoother and the filter, respectively. These representative conditional distributions are constructed using the same event weights, with the uncertainty incorporating both the within-event posterior uncertainty and the between-event variability according to the decomposition in \eqref{eq:total_uncertainty_prop}. The largest discrepancy appears during the growth stage prior to the peak ($t=0$): the smoother identifies an earlier rise of $\gamma$ with smaller uncertainty, whereas the filter responds later because it only uses information available up to the current time. Panel (f) quantifies these differences through the relative entropy. Rather than averaging event-wise relative entropies, the relative entropy is evaluated directly between the representative Gaussian distributions constructed above. The red and cyan curves measure the departure of the representative filter-based and smoother-based conditional distributions, respectively, from the full marginal distribution of $\gamma$, while the dashed blue curve compares the representative smoother and filter distributions directly. The pronounced pre-peak separation shows that retrospective inference captures precursor signatures that are only partially accessible in real time. In contrast, the gap becomes much smaller after the peak, indicating that once the event is fully developed, both estimators recover nearly the same hidden-state information. At the event peak, there is a brief increase in the discrepancy between the smoother and filter, since the smoother incorporates future information associated with the subsequent decay phase whereas the filter does not.

\begin{figure}[t]
\centering
\hspace*{-0cm}\includegraphics[width=1\linewidth]{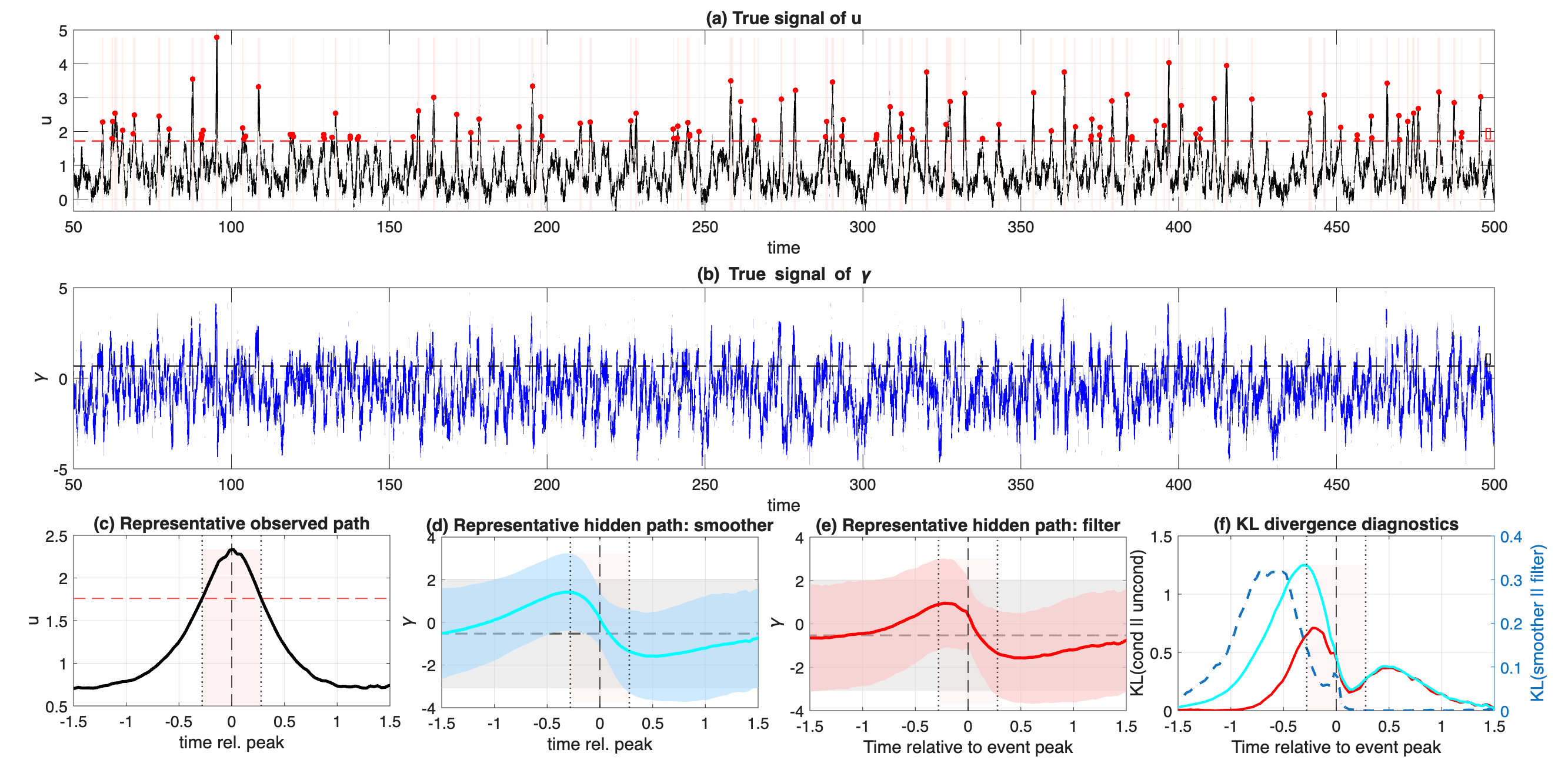}
\caption{Simulation and hidden-state mechanisms of extreme events in system \eqref{SPEKF_M_num}.
Panel (a): True signal of $u$, where red markers indicate events whose peak amplitudes exceed the 90th percentile and the red shading highlights the duration the extreme threshold is exceeded.
Panel (b): True signal of the stochastic damping $\gamma$. The dashed horizontal line marks the threshold $d_u/c$; crossing this level causes the net damping in the $u$-equation to become positive (anti-damping), enabling transient amplification.
Panel (c): Representative trajectory of $u$ conditioned on extreme events, aligned by the event peak at relative time $0$.
Panel (d): Most probable hidden trajectory conditioned on extreme events inferred by smoothing aligned to the event peak. The colored band denotes $\pm 2$ standard deviations around the posterior mean, while the gray shading represents the unconditional marginal variability of $\gamma$.
Panel (e): Same as in Panel (d), but inferred by filtering.
Panel (f): Relative entropy diagnostics. The red (cyan) curve shows the divergence between the filter-based (smoother-based) conditional distribution and the full marginal distribution of $\gamma$ (left axis). The dashed blue curve shows the divergence between the smoother and filter conditional distributions (right axis). Overall, the extreme events are associated with intermittent excursions of the hidden damping variable into the anti-damping regime, with the smoother identifying this hidden structure earlier than the filter due to its access to future data where the extreme episode is observed.}\label{Fig:SPEKF_M_Paths}
\end{figure}

Figure~\ref{Fig:SPEKF_M_Events} presents three event-level case studies of the system \eqref{SPEKF_M_num}: two strong extremes and one moderate extreme. These examples reveal how individual events emerge from hidden-state dynamics that are only partially visible in the observed signal. They also illustrate that distinct pathways can arise from the same underlying mechanism. In this example, the common anti-damping mechanism can manifest through different latent evolutions, including differences in the timing, strength, and persistence of the excursion into the anti-damping regime. Consistent with the statistical results, the smoother responds at the onset stage of each event (identified via \eqref{t_on}), substantially earlier than the filter. At that time, the observed trajectory $u$ often shows only weak or ambiguous changes, while the hidden variable $\gamma$ has already crossed the threshold $d_u/c$, initiating a period of positive effective damping (anti-damping) and cumulative growth. This demonstrates that conditioning on the realized event allows retrospective inference to recover precursor signatures that are difficult to identify in real time.
The last row further quantifies how hidden-state forcing at different times contributes to the eventual extreme event through the integrated influence ranges. Positive excursions of $\gamma$ above the threshold (anti-damping) can maintain a persistent influence from the onset through the event peak, whereas its contribution decays more rapidly after the peak. The two strong events are associated with intense anti-damping bursts that produce rapid amplification of $u$. In contrast, the moderate event is generated by a weaker but more sustained anti-damping episode, leading to a slower build-up and a longer interval between onset and peak. These case studies highlight that extreme-event magnitude depends not only on the instantaneous strength of the hidden trigger, but also on its duration and temporal persistence.

\begin{figure}[t]
\centering
\hspace*{-0cm}\includegraphics[width=1\linewidth]{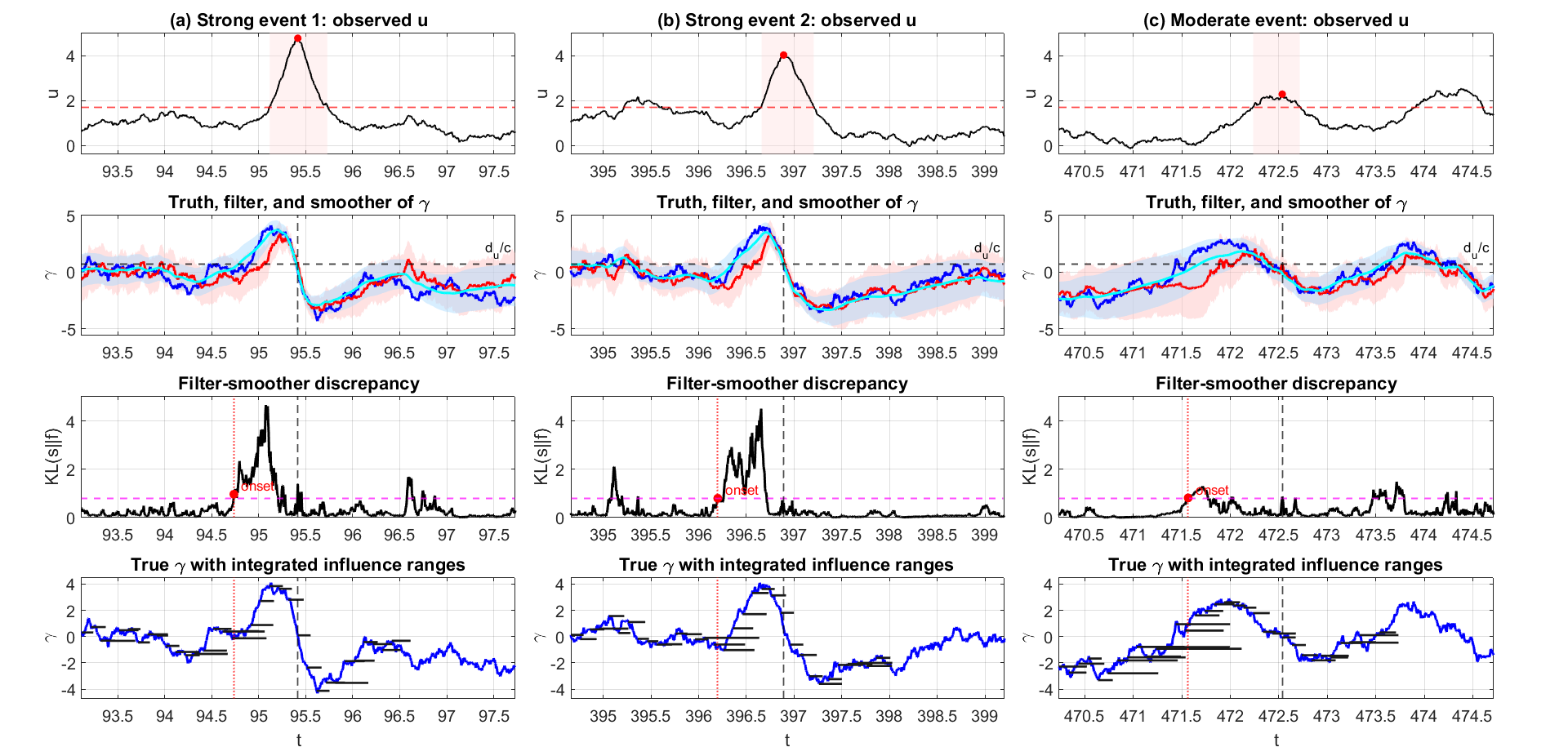}
\caption{Mechanisms of representative individual extreme events in system \eqref{SPEKF_M_num}.
The three columns correspond to two strong events and one moderate event.
First row: true signal of $u$ near each event. The event peak is marked by the red dot, and the red dashed horizontal line denotes the 90th-percentile threshold used to define extremes.
Second row: true hidden trajectory $\gamma$ (blue), together with the filter estimate (red) and smoother estimate (cyan). Shaded bands indicate $\pm2$ standard deviations around the corresponding posterior means. The horizontal dashed line marks the threshold $d_u/c$, above which the net damping in the $u$-equation becomes positive (anti-damping) and transient growth is induced. The vertical dashed line indicates the event peak time.
Third row: Relative entropy between the smoother and filter conditional distributions as a function of time. The horizontal magenta dashed line is the detection threshold $\kappa = 0.8$. The onset time is defined as the first threshold crossing within the $T_{pre} = 1.2$ time units preceding the peak.
Fourth row: true trajectory of $\gamma$ together with the integrated influence ranges (horizontal black segments), which quantify the time intervals over which hidden-state anomalies contribute to the event amplitude at the peak. Overall, the event-wise analysis shows how a common anti-damping mechanism can produce different extreme-event trajectories by resolving when the hidden precursor first becomes dynamically relevant and how long it influences the eventual peak.
}\label{Fig:SPEKF_M_Events}
\end{figure}

To summarize, this example illustrates how hidden-state dynamics govern intermittent extremes and how these mechanisms can be diagnosed through filtering, smoothing, and information-theoretic tools. Extreme events in the observed variable $u$ are systematically linked to positive excursions of the hidden damping variable $\gamma$ beyond the threshold $d_u/c$, where anti-damping mechanisms trigger $u$'s transient growth. The trajectory-wise and statistical analyses provide a consistent mechanistic interpretation in this relatively homogeneous example, while emphasizing complementary aspects of the hidden dynamics. The statistical analysis identifies the common anti-damping precursor structure across the extreme-event population, whereas the event-wise analysis resolves variations in its timing, strength, and persistence among individual realizations. A key distinction emerges between online and retrospective inference: smoothing detects the precursor phase earlier than filtering, showing that the onset of an extreme event is often invisible in the observed time series even though the hidden mechanisms have already developed. From a prediction perspective, relying only on the observed signal may therefore be insufficient, and more comprehensive predictive models that account for latent dynamics are needed. The event-wise case studies further show that extreme-event magnitude depends on both the strength and persistence of anti-damping episodes. Strong bursts produce rapid and pronounced extremes, whereas weaker but longer-lasting episodes generate moderate events through gradual accumulation. The estimated influence ranges confirm that precursor anomalies can affect the eventual peak over an extended time window, with sustained anti-damping contributing cumulatively up to the event peak.

\subsection{Multiple pathways to extreme events: damping, forcing, and their interactions}

We next consider an extended model with two hidden variables \cite{chen2023stochastic}:
\begin{equation}\label{SPEKF_MA_num}
\begin{split}
\frac{\d u}{\d t} &= -d_u u +c\gamma u + F_u + b + \sigma_u \dot{W}_u,\\
\frac{\d \gamma}{\d t} &= -d_\gamma \gamma - cu^2 + F_\gamma + \sigma_\gamma \dot{W}_\gamma,\\
\frac{\d b}{\d t} &= -d_b b + \sigma_b \dot{W}_b.
\end{split}
\end{equation}
The parameters used in this model are as follows:
\begin{equation}
\begin{gathered}
d_u=0.8,\qquad d_\gamma=0.8,\qquad d_b=1,\qquad c=1.2,\\
F_u=1,\qquad F_\gamma=1,\qquad
\sigma_u=0.5,\qquad \sigma_\gamma=2,\qquad \sigma_b=2.5,
\end{gathered}
\end{equation}
with numerical integration time step $\Delta t=0.005$. A long trajectory is generated, and extreme events are identified from the upper and lower $5\%$ tails of the post-transient distribution of $u$ using a minimum peak-separation constraint.

In addition to the stochastic damping $\gamma$, the model \eqref{SPEKF_MA_num} now includes a stochastic forcing variable $b$. The observable mode $u$ can therefore experience extreme growth through several distinct mechanisms: weakened damping, strong positive forcing, or a cooperative interaction between both effects. In addition, this model admits not only positive extremes but also negative extreme episodes, as $F_u+b$ can vary in sign over time with $\sigma_b = 2.5$ unlike $F_u = 1 > 0 $ in \eqref{SPEKF_M_num}. This setting is designed to study heterogeneity among extreme events. In many realistic systems, observed extremes do not arise from a single universal pathway, and events with similar amplitudes may be generated by substantially different hidden processes.

To characterize these pathways, each detected event is represented by a feature vector containing peak amplitude, duration, pre-event growth rate, hidden-state averages, and integrated damping/forcing contributions over the precursor window; see Appendix A for a complete list. After standardization, $k$-means clustering is applied to the feature space, and three robust clusters are identified. The choice of the number of clusters was motivated by the three-group organization apparent in the principal-component projection on the chosen feature space, together with the dynamical interpretability of the resulting groups (damping-induced, forcing-driven, and mixed). These clusters reveal distinct candidate pathway classes rather than collapsing the heterogeneous extreme-event population into a single representative evolution.

\begin{figure}[t]
\centering
\hspace*{-0cm}\includegraphics[width=1\linewidth]{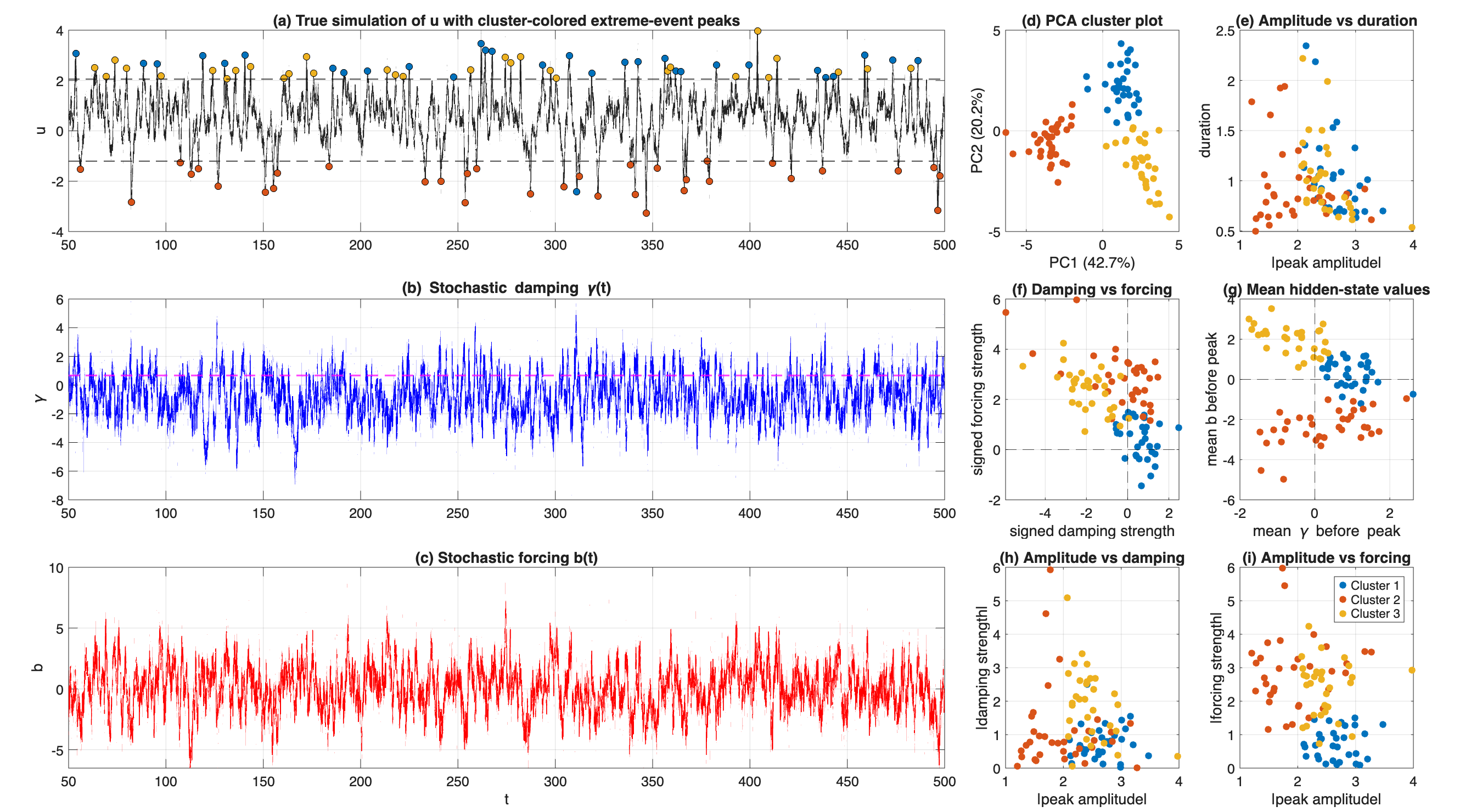}
\caption{
Simulation of system~\eqref{SPEKF_MA_num} and cluster diagnostics for the detected extreme events.
Panel (a): True trajectory of $u$, where colored markers indicate event peaks assigned to different clusters (Cluster 1 blue, Cluster 2 red, and Cluster 3 yellow); dashed horizontal lines denote the positive and negative extreme-event thresholds.
Panel (b): Hidden stochastic damping $\gamma(t)$. The dashed horizontal line marks the anti-damping threshold above which the net linear damping in the $u$-equation becomes positive.
Panel (c): Hidden stochastic forcing $b(t)$.
Panel (d): Projection of the adopted event feature vectors onto the first two principal components, showing clear cluster separation (see Appendix A).
Panel (e): Peak amplitude versus event duration.
Panel (f): Signed integrated damping contribution versus signed integrated forcing contribution during the precursor stage.
Panel (g): Mean pre-peak values of $\gamma$ and $b$.
Panel (h): Peak amplitude versus damping strength.
Panel (i): Peak amplitude versus forcing strength. Together, these diagnostics show that events with similar observed amplitudes can arise through distinct dynamical combinations of hidden damping and forcing, yielding clearly separated pathway classes.
}
\label{Fig:SPEKF_MA_Simulation}
\end{figure}

Figure \ref{Fig:SPEKF_MA_Simulation} summarizes the full simulation and the geometry of the clustered events. Panel~(a) shows that positive and negative extremes coexist in the same observed time series, while the cluster labels indicate that similar peak amplitudes can arise from different hidden pathways. Panels (b)--(c) display the two latent drivers, $\gamma$ and $b$, whose intermittent excursions generate bursts of growth or suppression in $u$. The principal component analysis (PCA) projection in panel (d) confirms that the extracted event features naturally organize into separated groups. The remaining scatter plots provide physical interpretation: some events are associated with stronger damping-related contributions, some with forcing-dominated contributions, and others with mixed signatures involving both variables.

\begin{figure}[t]
\centering
\hspace*{-0cm}\includegraphics[width=1\linewidth]{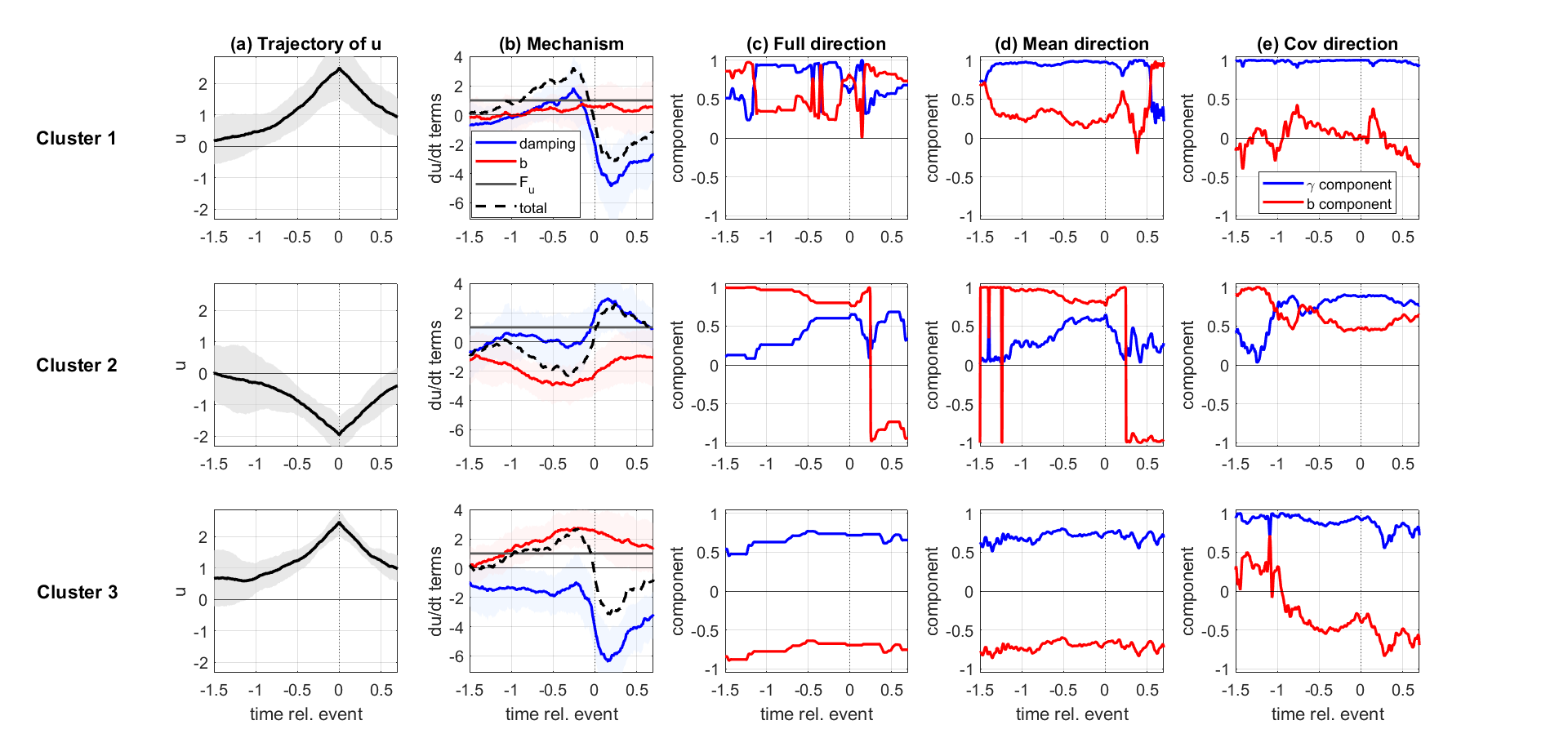}
\caption{
Cluster-wise representative pathways for the extreme events in system~\eqref{SPEKF_MA_num}. Each row corresponds to one cluster.
First column: aligned mean trajectory of $u$ relative to the event peak ($\tau=0$), with shading indicating one standard deviation across events in the cluster.
Second column: mean decomposition of the $u$ tendency found by averaging the value of each term across extreme events with the event peak aligned to relative time 0: the damping-related term $(-d_u+c\gamma)u$ (blue), stochastic forcing $b$ (red), constant forcing $F_u$ (gray), and their total contribution (black dashed). Shaded regions indicate one standard deviation for the damping and forcing terms.
Third column: most sensitive hidden direction based on the full information criterion.
Fourth column: sensitive direction induced by mean differences only.
Fifth column: sensitive direction induced by covariance differences only. In the last three columns, blue and red curves denote the $\gamma$- and $b$-components of the unit direction vectors, respectively. In all panels the vertical dashed line marks the event peak time at relative time $\tau = 0$. Overall, the three clusters exhibit distinct dynamical pathways, with different roles of damping and forcing in driving the event evolution and different hidden-state signatures in both their means and uncertainties.}
\label{Fig:SPEKF_MA_Cluster}
\end{figure}

Figure \ref{Fig:SPEKF_MA_Cluster} provides a dynamical interpretation of the three clusters. The first column compares the typical observed trajectories. From Panel (a) of Figure \ref{Fig:SPEKF_MA_Simulation}, Clusters 1 and 3 correspond to positive extreme events, whereas Cluster 2 corresponds to negative extreme events. Cluster~1 exhibits sharper and shorter bursts, with a more rapid build-up toward the peak. In contrast, Clusters 2 and 3 show broader events with slower growth and recovery. The second column decomposes the growth mechanisms and clarifies the distinct triggering pathways. Cluster 1 is primarily damping driven: anti-damping supplies the dominant growth mechanism and leads to faster amplification than in Cluster 3, which is mainly forcing driven through bursts in $b$. Cluster 2 is also largely forcing driven, while damping further enhances the growth once the event is initiated. In particular, the negative forcing is crucial for generating this class of negative extremes. Without sufficiently strong forcing of opposite sign, the events would remain predominantly positive, as in Sec.~IV~A where only stochastic damping was present. Thus, events with comparable amplitudes need not share the same internal cause.

The last three columns quantify the most informative hidden directions that distinguish each cluster from the background behavior. For the full sensitive direction, Clusters 1 and 2 show cooperative contributions from damping and forcing, so the $\gamma$- and $b$-components tend to point in the same effective direction toward the corresponding extremes. In contrast, Cluster 3 is primarily forcing driven: the forcing term enhances the positive events, while damping acts in the opposite direction and partially suppresses their growth. For Cluster 1, the contribution of damping is relatively weak during the early precursor stage. However, once the total damping crosses the instability threshold, as seen in the mechanism decomposition, damping rapidly becomes the leading driver and dominates the final amplification toward the peak.  Consistent with panel (b) for Cluster 2, while damping does play a role in exciting extreme behavior, forcing is the most influential throughout, where the episodes are primarily built up by forcing prior to the peak, and then strongly suppressed after. For Cluster 3, throughout the lead-up, peak and decay of the extreme episode the roles of forcing and damping remain the same. The full sensitive direction is broadly consistent with the direction obtained from mean differences alone, indicating that shifts in the hidden-state averages provide the main separation between clusters. The covariance-based direction, however, places stronger weight on the damping component across all clusters, since anti-damping episodes can trigger rapid growth and substantially enlarge event uncertainty. This reflects the covariance contribution in the relative entropy. Overall, the relevant precursor variable is mechanism dependent and may evolve over the event life cycle, while a full contribution profile of the distribution, both in mean and uncertainty, provides a comprehensive understanding of the underlying mechanisms. Hence, predictive models based only on the observed signal $u$ or on a single hidden indicator may miss important classes of extremes.

To summarize, in contrast to the previous example, where different latent pathways arose from a common anti-damping mechanism, the present example demonstrates that introducing multiple hidden variables can lead to genuinely distinct pathways associated with different underlying dynamical processes. Events with similar observed amplitudes may arise from anti-damping, stochastic forcing, or cooperative interactions between both mechanisms, and these pathways produce different build-up, duration, and recovery characteristics. The clustering analysis shows that such heterogeneity can be identified systematically from trajectory-based and statistical features rather than inferred from the observed signal alone. Within the resulting pathway classes, the trajectory-wise and distribution-based diagnostics provide broadly consistent physical interpretations while emphasizing different aspects of the hidden dynamics. The trajectory-wise decomposition identifies the dominant contributions governing the temporal evolution of each event class, whereas the distribution-based sensitive-direction analysis characterizes how the hidden-state statistics are reorganized and can reveal additional changes in uncertainty that are not apparent from the mean pathway alone. The sensitive-direction diagnostics further reveal that the dominant precursor variable depends on the event class and can evolve over time. From a prediction perspective, this indicates that reliable forecasts require models that account for variation among latent pathways rather than relying only on the observable mode or a single averaged indicator.

\subsection{Regime transitions and extreme unblocking in topographic flow models}\label{Sec:Topo}

Large-scale atmospheric blocking and unblocking events are prototypical examples of regime transitions in geophysical fluid dynamics \cite{brunner2017connecting, vallis2017atmospheric}. Blocking corresponds to persistent, quasi-stationary large-scale flow anomalies that deflect the mean westerlies and can generate prolonged surface extremes such as heatwaves, cold-air outbreaks, droughts, and heavy precipitation. Rapid exits from blocked states (unblocking) are equally important, since they determine the termination and duration of these impacts. A central scientific question is how nonlinear interactions among the large-scale flow, smaller eddies, stochastic forcing, and spatial topography trigger such transitions.

We consider a barotropic quasi-geostrophic flow model over topography \cite{qi2018predicting, majda2006nonlinear}. Let $\psi(x,y,t)$ denote the streamfunction and define the potential vorticity
\begin{equation}
q=\Delta\psi+h(x,y),
\end{equation}
where $h(x,y)$ is the bottom topography. The governing stochastic equation is
\begin{equation}\label{eq:topo_pde_full}
\partial_t q + J(\psi,q)+\beta \partial_x \psi
=
-\nu \Delta \psi+\sigma_\psi \dot W_\psi,
\end{equation}
where $J(a,b)=a_x b_y-a_y b_x$ is the Jacobian operator, $\beta$ is the planetary vorticity gradient (beta-plane approximation), $\nu$ is a damping coefficient, and $\dot W_\psi$ is stochastic forcing. The damping and stochastic terms parameterize unresolved smaller-scale processes, neglected physics, and model uncertainty. Topography breaks translational symmetry, introduces preferred flow states and anisotropy, and provides a key source of multiple regimes and low-frequency transitions.

A layered truncation can be constructed by separating the zonal mean flow from a collection of smaller Fourier modes, which eliminates the Jacobian term. With an appropriate partition into observed and hidden variables, the resulting model belongs to the conditional Gaussian modeling framework introduced earlier \cite{chen2018conditional, chen2024physics}. Therefore, the filtering, smoothing, and conditional Gaussian mixture tools developed in Sections \ref{Sec:Framework}--\ref{Sec:Mechanisms} apply directly. Since our emphasis here is on a broader setting for extreme events studies, we only briefly acknowledge that structured case and focus below on the general model.

To this end, we Fourier expand the streamfunction over the full spatial domain as
\begin{equation}
\psi(x,y,t)
=
- V(t)y
+
\sum_{\mathbf{k}\in\mathcal K}
\phi_{\mathbf{k}}(t)e^{i(k_xx+k_yy)},
\qquad
\mathbf{k}=(k_x,k_y),
\end{equation}
where $V(t)$ is the large-scale zonal mean flow and $\phi_{\mathbf{k}}(t)$ are complex modal amplitudes satisfying the reality condition
$
\phi_{-\mathbf{k}}=\phi_{\mathbf{k}}^\ast.
$
The first term represents the dominant jet component, while the Fourier modes describe small-scale eddies and wave disturbances interacting with the mean flow.

Substituting the expansion into \eqref{eq:topo_pde_full} yields a finite-dimensional stochastic system after a truncation to a finite set of wavenumbers $\mathcal K$,
\begin{subequations}\label{eq:topo_fourier_full}
\begin{align}
\d \phi_{\mathbf{k}}
&=
\Bigg[
\sum_{\mathbf{m}\in\mathcal K}
\frac{C(\mathbf{k},\mathbf{m})}{|\mathbf{k}|^2}\,
\phi_{\mathbf{k}-\mathbf{m}}
\big(-|\mathbf{m}|^2\phi_{\mathbf{m}}+h_{\mathbf{m}}\big)
+i k_x\Big(\frac{\beta}{|\mathbf{k}|^2}-V\Big)\phi_{\mathbf{k}}
+i k_x \frac{h_{\mathbf{k}}}{|\mathbf{k}|^2}V
-\nu \phi_{\mathbf{k}}
\Bigg]\d t\notag\\&\qquad\qquad
+\sigma_\psi \d W_{\mathbf{k}},
\\
\d V
&=
\Bigg[
\sum_{\mathbf{k}\in\mathcal K}
k_x\,\mathrm{Im}\!\big(h_{\mathbf{k}}\phi_{\mathbf{k}}^\ast\big)
-\nu_VV
\Bigg]\d t
+\sigma_V \d W_V.
\end{align}
\end{subequations}
Here the quadratic nonlinear interaction coefficient is
\begin{equation}
C(\mathbf{k},\mathbf{m})
=
-\big(\mathbf{k}^\perp-\mathbf{m}^\perp\big)\cdot \mathbf{m},
\qquad
\mathbf{k}^\perp=(-k_y,k_x),
\end{equation}
and the convention $\phi_{\mathbf{k}-\mathbf{m}}=0$ is used whenever $\mathbf{k}-\mathbf{m}\notin\mathcal K$. Thus, the first summation in the $\phi_{\mathbf{k}}$ equation represents triad interactions among retained Fourier modes generated by nonlinear advection. Each mode evolves through energy exchange with pairs of interacting modes, together with Rossby-wave propagation, advection by the zonal mean flow, topographic forcing, damping, and stochastic excitation.

The drift term in the $V$ equation is the feedback from the wave field and topography onto the zonal mean flow. It is obtained from the Reynolds-stress form of the truncated dynamics and measures the aggregate momentum transfer from the eddies to the large-scale jet. The parameters $\nu$ and $\nu_V$ are damping coefficients for the eddies and mean flow, while $\sigma_\psi$ and $\sigma_V$ denote stochastic forcing amplitudes representing unresolved variability and model uncertainty.

Unlike the layered case, system \eqref{eq:topo_fourier_full} is not conditionally Gaussian because the hidden modal variables interact fully nonlinearly. Therefore, the closed-form filtering and smoothing equations are unavailable. If trajectory-wise state reconstruction is desired, one may instead use ensemble-based data assimilation methods such as the ensemble Kalman filter or ensemble smoothers. In this subsection, however, we focus on the statistical viewpoint and use direct Monte Carlo simulation.

The numerical experiments use the nondimensional parameters
\begin{equation}
\gamma=0.18,\qquad
\beta=5\sqrt{\gamma},\qquad
\nu=0.02,\qquad
\nu_V=0.01,\qquad
\sigma_\psi=0.03,\qquad
\sigma_V=0.015.
\end{equation}
These values place the system in a regime with weak damping, sustained stochastic excitation, and intermittent transitions between multiple metastable states. The resolved Fourier mode set is
\begin{equation}
\mathcal K=\mathcal K' \cup (-\mathcal K'), \quad
\mathcal K' = \{(0,2),(1,2),(0,1),(1,1),(2,1),(1,0),(2,0),(1,-1),(2,-1),(1,-2)\}.
\end{equation}
The retained Fourier modes represent a small set of large- and intermediate-scale wave patterns that interact to produce regime transitions. The zonal modes $(1,0)$ and $(2,0)$ describe streamwise ridges and troughs that directly modulate the mean jet structure. The meridional modes $(0,1)$ and $(0,2)$ represent north-south meanders and latitudinal shifts of the flow. The diagonal and near-diagonal modes $(1,\pm1)$ and $(2,\pm1)$ correspond to tilted wave patterns that transport momentum and vorticity across the domain and are particularly effective in triggering transitions between blocked and unblocked states. Higher-wavenumber members within each family provide sharper spatial gradients and more localized deformation. Although highly truncated, this collection captures the minimal set of interacting wave geometries needed to generate multiple metastable regimes, intermittent transitions, and nontrivial precursor structures that can lead to various distinct classes or clusters of latent-based mechanisms for generating extreme events.
The topography is prescribed through its Fourier coefficients $h_{\mathbf{k}}$. On the positive wavevectors we choose
\begin{equation}
h_{(1,1)}=0, \quad
h_{(2,1)}=-0.75\gamma, \quad
h_{(1,0)}=0.15\gamma, \quad
h_{(2,0)}=1.55\gamma, \quad
h_{(1,-1)}=0.90\gamma, \quad
h_{(2,-1)}=1.65\gamma,
\end{equation}
with all remaining coefficients zero and with symmetry $h_{-\mathbf{k}}=h_{\mathbf{k}}$.
Thus,
\begin{equation}
h(x,y)=\sum_{\mathbf{k}\in\mathcal K} h_{\mathbf{k}} e^{i(k_xx+k_yy)}.
\end{equation}
This creates an anisotropic terrain pattern combining zonal ridges, tilted ridge-valley structures, and competing oblique components. Because the strongest amplitudes occur in the zonal and lower-diagonal families (e.g., $(1,-1)$, $(2,-1)$), the flow experiences several preferred standing-wave responses rather than a single stationary pattern. This promotes blocking, unblocking, and irregular regime switching.

We generate a long Monte Carlo simulation of \eqref{eq:topo_fourier_full} and treat the large-scale quantity $V(t)$ as the observed variable, while the small-scale Fourier modes $\phi_{\mathbf{k}}$ play the role of hidden variables. Extreme events are identified using large values of $V$, corresponding to strong episodes with zonal jets. We then compare hidden-variable statistics conditioned on these events with the full equilibrium statistics. To classify the topographic-flow extremes, each detected event is represented by a feature vector that combines observable trajectory characteristics with hidden-wave diagnostics. The features include the event peak amplitude, pre- and post-event growth rates, local variability of the large-scale flow index $V$, total hidden modal energy, energy partition among zonal, meridional, and diagonal mode families, several anisotropy ratios, and selected precursor or peak-time Fourier-mode amplitudes. This construction allows events with similar observed magnitudes but different hidden precursor organizations to be distinguished systematically through unsupervised clustering. Full definitions of all feature variables are provided in Appendix B.

\begin{figure}[t]
\centering
\hspace*{-0cm}\includegraphics[width=1\linewidth]{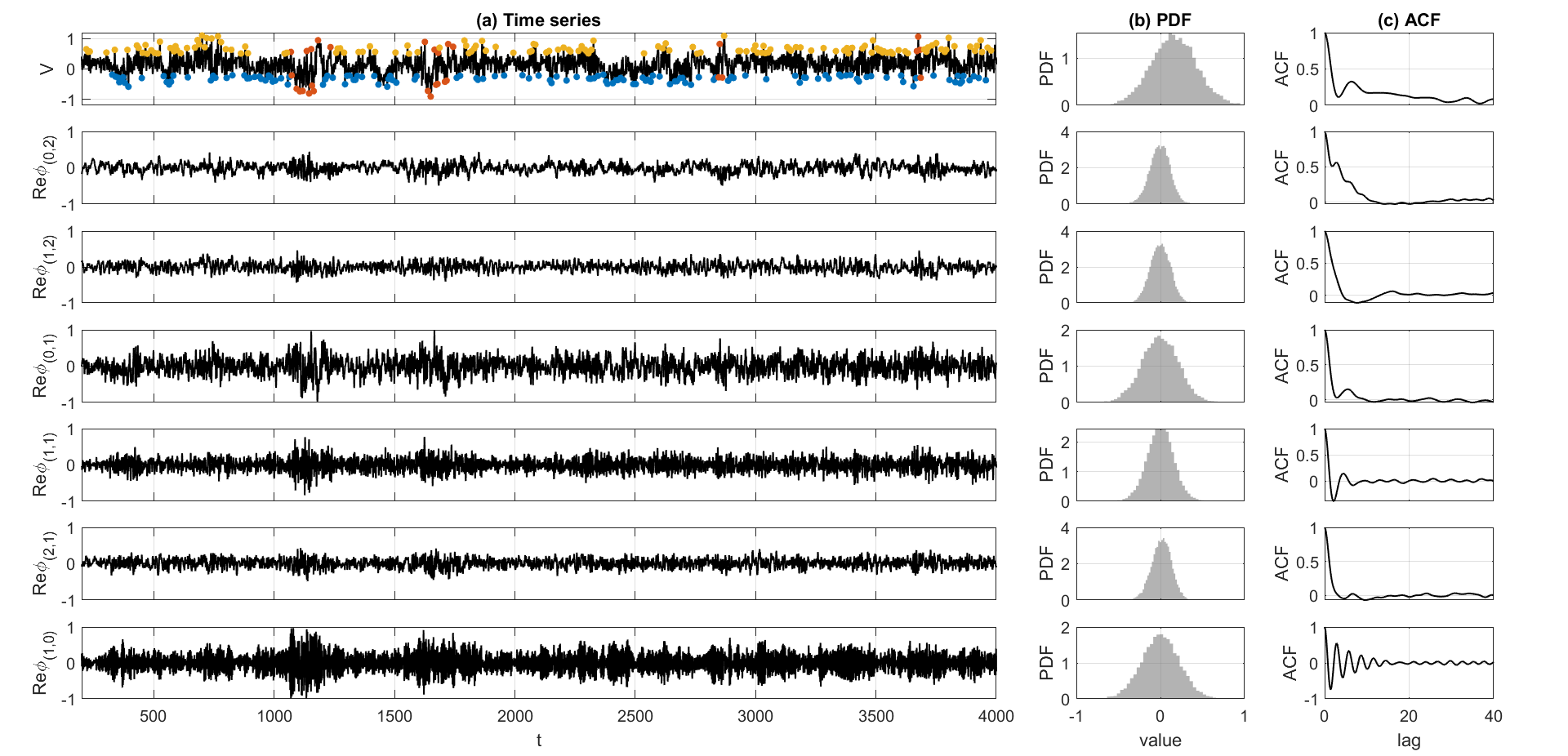}
\caption{
Basic statistical properties of the topographic-flow model~\eqref{eq:topo_fourier_full}. Each row corresponds to the observed large-scale variable $V$ or to one representative hidden Fourier mode $\mathrm{Re}(\phi_{\mathbf{k}})$.
(a) Time series from the post-transient regime. Colored markers on the $V$ panel indicate detected extreme events and their respective cluster labels (Cluster 1 blue, Cluster 2 red, and Cluster 3 yellow).
(b) Empirical probability density functions. The distribution of $V$ is broad and weakly non-Gaussian, reflecting intermittent transitions between preferred flow states, while the modal amplitudes $\mathrm{Re}(\phi_{\mathbf{k}})$ remain centered near zero with different variances reflecting their relative excitation levels.
(c) Autocorrelation functions. The variable $V$ exhibits pronounced low-frequency memory compared with the faster modal variables, whose correlations decay more rapidly and may display weak oscillatory signatures associated with wave dynamics. Overall, the statistics reveal distinct roles for the observed and hidden variables: \(V\) captures the persistent blocked--unblocked regime behavior, while the hidden Fourier modes exhibit transient bursts of wave activity and fluctuations with substantially shorter memory.
}
\label{Fig:Topo_Time_Series_PDF_ACF}
\end{figure}

Figure~\ref{Fig:Topo_Time_Series_PDF_ACF} summarizes the baseline behavior of the truncated topographic model. The first row shows that the zonal-mean flow $V$ undergoes intermittent excursions between blocked and unblocked regimes, motivating its use as the observable quantity for defining extreme events. In contrast, the hidden Fourier modes fluctuate around zero with smaller amplitudes but with strongly time-dependent variance, indicating bursts of wave activity during regime transitions. The PDF column highlights the separation between the slow large-scale variable and the faster eddies: $V$ carries the broad regime structure, whereas some of the individual modes remain more localized around the equilibrium. The ACF column further emphasizes this scale separation. The slow decay of the $V$ autocorrelation indicates persistent metastable behavior, while the more rapid decorrelation of the modal amplitudes is consistent with transient eddy forcing and wave interactions.

\begin{figure}[t]
\centering
\hspace*{-0cm}\includegraphics[width=1\linewidth]{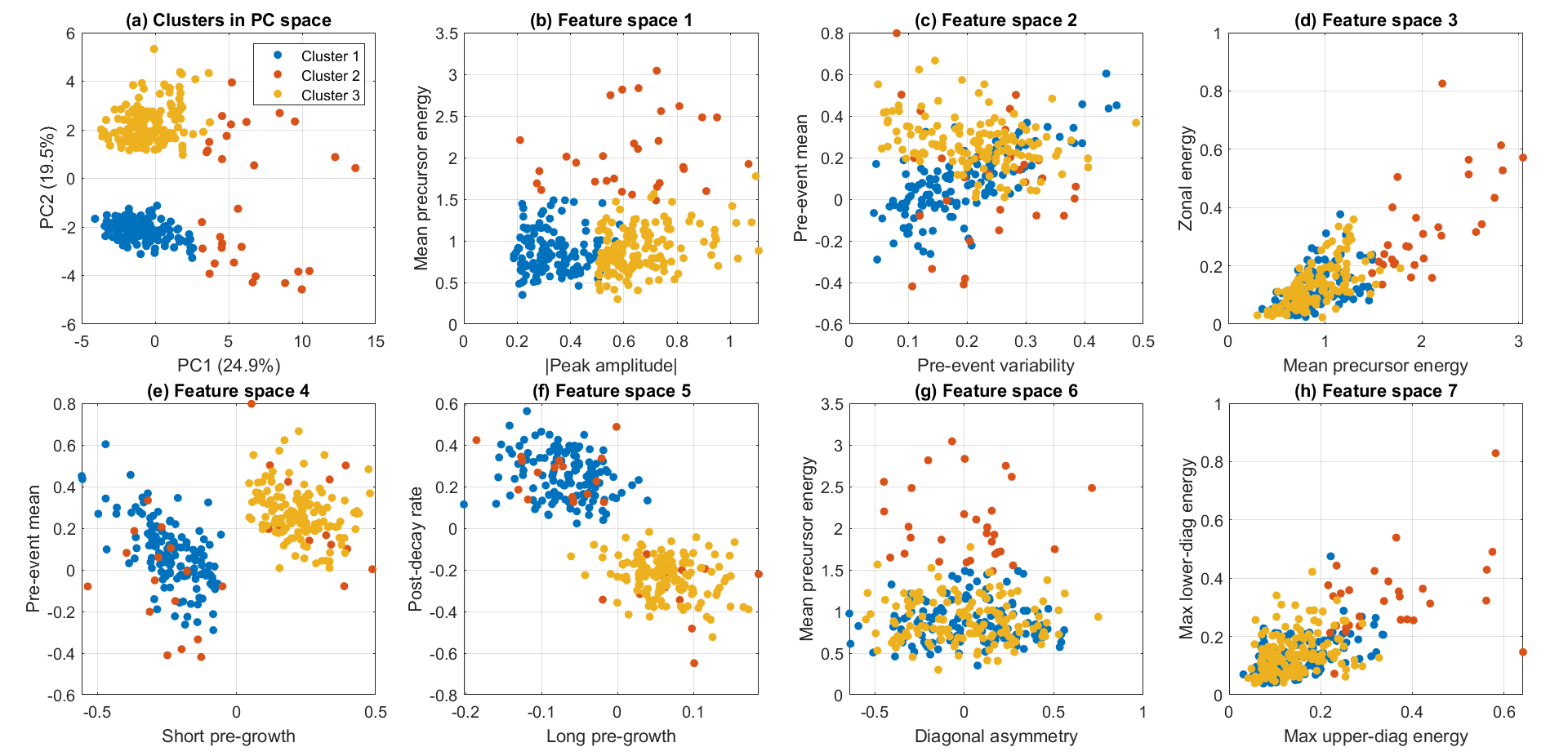}
\caption{
Mechanism-based clustering of blocking and unblocking events in the topographic-flow model (Cluster 1 blue, Cluster 2 red, and Cluster 3 yellow).
Panel (a): Extreme events projected onto the first two principal components of the feature space. Three well-separated clusters emerge, indicating distinct classes of transitions.
Panels (b)--(h): Two-dimensional projections onto a selection of the adopted physical features (see Appendix B), including event amplitude, precursor variability, pre-event mean flow, growth and decay slopes, and modal-energy diagnostics. These coordinates reveal how the clusters differ in both observable signatures and hidden precursor structures. In particular, some clusters are associated with stronger hidden energy build-up, whereas others are characterized by sharper transitions or different balances among modal families.
}
\label{Fig:Topo_Clusters_2D}
\end{figure}

Figure \ref{Fig:Topo_Clusters_2D} shows that the detected extremes are not generated by a single universal transition pathway. Instead, three coherent classes emerge in the feature space. The principal-component projection in Panel~(a) already displays clear separation, confirming that the selected trajectory and hidden-state diagnostics capture genuine dynamical differences among events. The remaining panels explain the physical meaning of this separation.
Clusters 1 and 3 are primarily distinguished by opposite precursor growth signatures and therefore correspond mainly to negative and positive extreme events, respectively. This is seen most clearly in Panels (e) and (f): Cluster 1 has negative pre-growth and positive post-decay tendencies, consistent with trajectories evolving toward negative excursions and then recovering, whereas Cluster 3 shows positive pre-growth and negative post-decay tendencies associated with positive extremes. In contrast, Cluster 2 is not characterized mainly by sign, but by substantially stronger hidden-wave activity.
Indeed, Cluster 2 has the largest mean precursor energy, as shown in Panels (b), (d), and (g), indicating that these events are preceded by much stronger excitation of the hidden Fourier modes, as is reflected by Panel (a) in Figure \ref{Fig:Topo_Time_Series_PDF_ACF}. It also occupies the upper range of the zonal-energy coordinate in Panel~(d), showing enhanced energy in the zonal modes. Moreover, Panel (h) shows that Cluster 2 attains the largest simultaneous amplitudes in both the upper-diagonal modes $((1,1),(2,1))$ and lower-diagonal modes $((1,-1),(2,-1))$ on the Fourier wavenumber lattice. This indicates a broader activation across multiple wave families rather than concentration in a single mode sector.
Panels (c)--(e) further show that Cluster 2 has larger variability and wider spread in pre-event flow statistics, consistent with more diverse transition pathways. Thus, Cluster 1 and Cluster 3 represent relatively coherent sign-dependent transitions in the large-scale mean flow, while Cluster~2 corresponds to energetically rich and structurally diverse events driven by stronger multiscale wave interactions.

\begin{figure}[t]
\centering
\hspace*{-0cm}\includegraphics[width=1\linewidth]{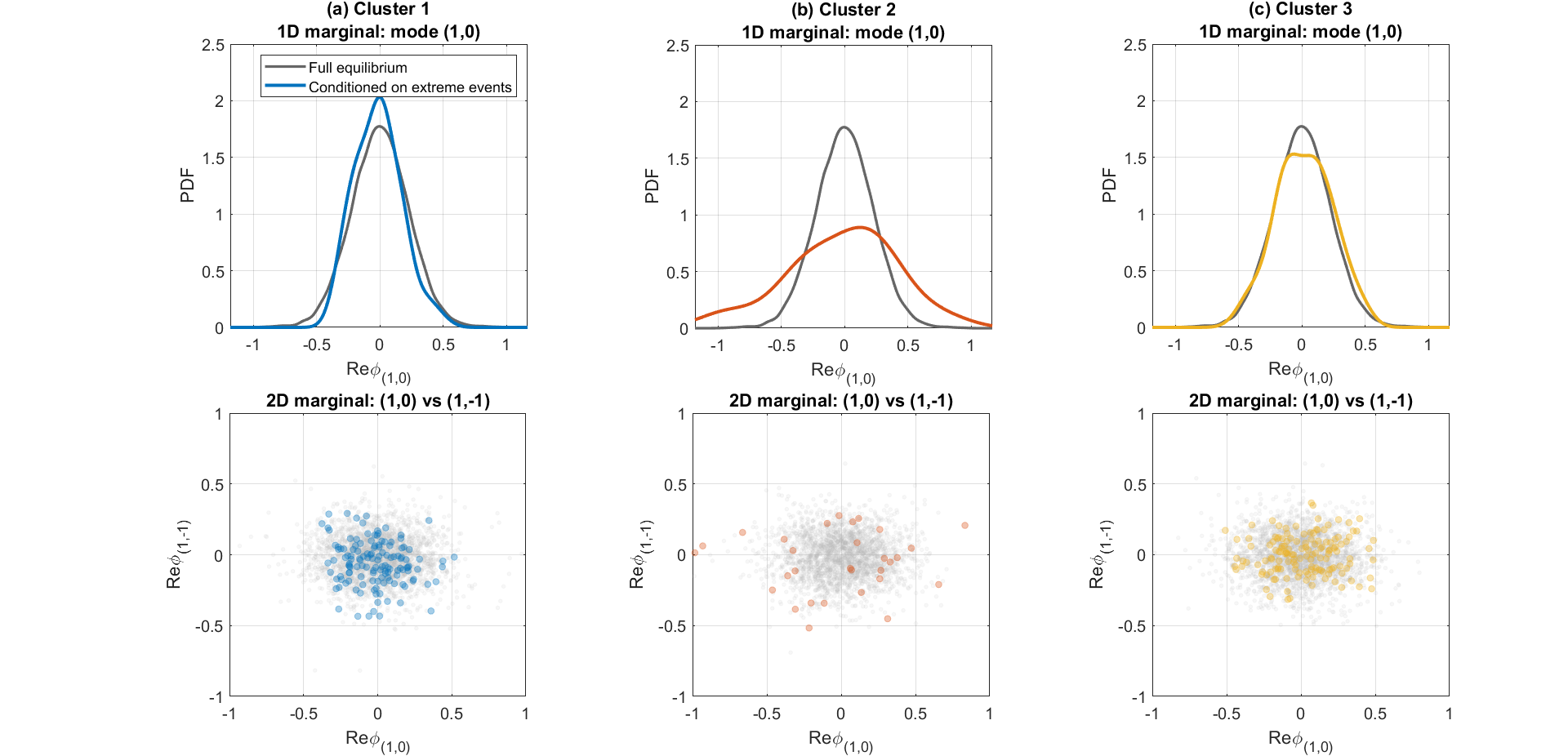}
\caption{
Hidden-state distributions conditioned on each event cluster [panels (a)--(c) for Clusters 1, 2, and 3, respectively], compared with the full equilibrium statistics. For each column, the top panel shows a one-dimensional marginal density of mode $\mathrm{Re}(\phi_{(1,0)})$, where the gray curve denotes the full climatological distribution and the colored curve denotes the distribution conditioned on that cluster. The other modes have a similar behavior. The bottom panel shows the corresponding two-dimensional marginal distribution of $(\mathrm{Re}(\phi_{(1,0)}),\mathrm{Re}(\phi_{(1,-1)}))$, where gray points represent the background attractor and colored points correspond to event times within the cluster. Overall, hidden precursors in the topographic model are not characterized by localization in a single latent variable, but instead by multivariate distributional organization and, for Cluster 2, enhanced uncertainty associated with intermittent hidden-wave activity.}
\label{Fig:Topo_Conditional_PDF}
\end{figure}

Figure \ref{Fig:Topo_Conditional_PDF} provides a more subtle picture of the hidden precursor structure than in the previous two examples. For Clusters 1 and 3, the conditional distributions are close to the full equilibrium marginals in both the one-dimensional and two-dimensional views. This contrasts sharply with the earlier low-dimensional models, where extreme events were associated with latent states concentrated in a localized region of hidden space. Here, the triggering mechanisms cannot be characterized by a simple geometric threshold or by occupying one special subset of the latent variables. Instead, the relevant information lies in coordinated combinations, relative phases, and energy partitions across multiple interacting Fourier modes. In other words, the precursor signatures are distributed over the multiscale hidden-state configuration rather than concentrated in a single dominant variable.
Cluster 2 exhibits a much more uncertain conditional distribution, with larger spread than the background equilibrium statistics. Although initially surprising, this behavior is consistent with the trajectory analysis in panel (a) of Figure \ref{Fig:Topo_Time_Series_PDF_ACF}, where events in this class are associated with intermittent bursts in the hidden modes and strongly variable wave activity. Thus, for this cluster, the precursor is not localization in state space but enhanced variability and episodic activation of the latent dynamics. These results highlight that in higher-dimensional systems, extreme-event precursors may appear through structured multivariate organization or intermittent uncertainty growth, rather than through simple mean shifts of individual hidden variables.

\begin{figure}[t]
\centering
\hspace*{-0cm}\includegraphics[width=1\linewidth]{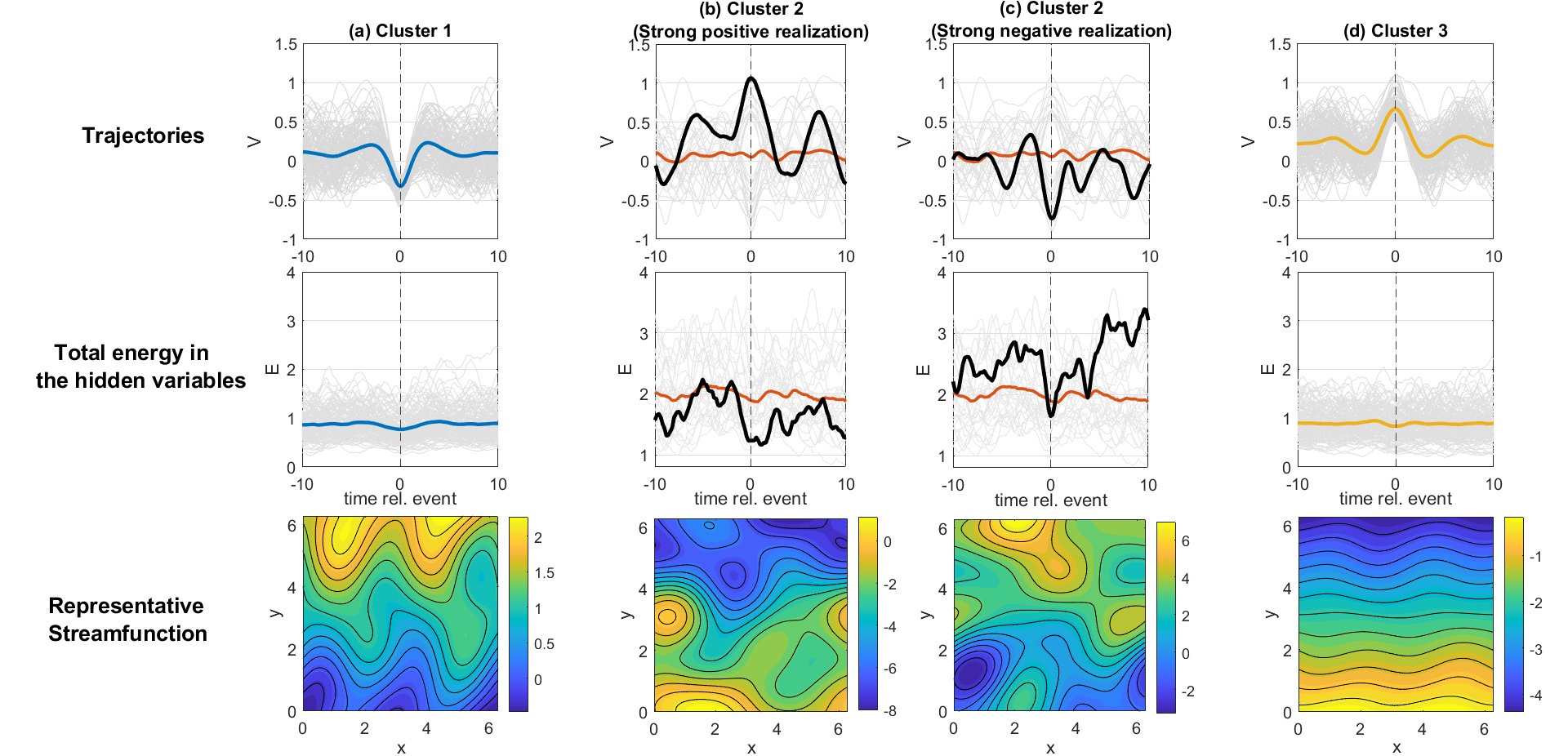}
\caption{
Representative trajectories, hidden-energy evolution, and flow patterns for the three event clusters [panel (a): Cluster 1; panels (b) and (c): Cluster 2; panel (d): Cluster 3]. Top row: aligned trajectories of the observed variable $V$ relative to the peak event time, where gray (and black in Cluster 2) curves denote individual events and colored curves denote cluster means. For Cluster~2, both a strong positive and a strong negative realization are shown, illustrating mixed-sign transitions within the same class. Middle row: total hidden modal energy $E=\sum_{\mathbf{k}}|\phi_{\mathbf{k}}|^2$ aligned in the same way. Bottom row: representative streamfunction fields reconstructed at event time from the cluster mean states for Clusters 1 and 3, and the streamfunction fields corresponding to the realization in black for Cluster 2. Together, these results show that the event classes differ not only in observed amplitude and timing, but also in precursor energy, variability, and the spatial flow structures through which the transitions develop.
}
\label{Fig:Topo_Clusters_Representatives}
\end{figure}

Figure \ref{Fig:Topo_Clusters_Representatives} provides an event-wise interpretation of the clustered transitions. The top row shows that the clusters possess distinct temporal signatures in the observed flow index $V$. Cluster 1 and Cluster 3 correspond primarily to negative and positive extreme events, respectively, whereas Cluster 2 contains both positive and negative realizations. Thus, the clustering is not determined solely by the sign of the event, but by the full dynamical structure of its precursor and transition pathway. In particular, Cluster 2 exhibits broader trajectories with stronger fluctuations and greater diversity across realizations.
The middle row shows that the total hidden modal energy evolves differently across clusters. Cluster 2 has substantially larger hidden energy not only near the event time but also during the precursor and recovery phases. This indicates that Cluster 2 consists of more energetic transitions with stronger variability, consistent with the wider spread of trajectories within that class. In contrast, Clusters 1 and 3 occur under weaker hidden-wave activity and are therefore dynamically more coherent.
The bottom row links these temporal signatures to distinct spatial flow structures. Because the mean forcing in $V$ is positive, Cluster 3 events correspond to strong positive zonal flow and display a clear jet pattern, representing canonical unblocking states. Cluster 1 events correspond to negative extremes with reversed or weakened zonal transport (blocked mean westerlies, meridional flow). Cluster 2 is more subtle: even when $V$ is strongly positive, the flow can still retain blocking-like structure because the hidden modes remain energetic. In particular, enhanced meridional variability (north-south waviness of the streamlines) and meandering jets can disrupt a purely zonal east-west current and generate complex mixed states. Thus, the framework separates not only event amplitudes, but complete dynamical pathways involving precursor energy growth, transition timing, and coherent flow geometry.

\section{Conclusion and Discussion}\label{Sec:Conclusion}

\subsection{Summary of this work}

This paper developed a mathematical framework for studying the mechanisms and pathways of extreme events in partially observed stochastic dynamical systems with hidden variables. The central objective was to move beyond analyses based solely on observed extremes and to infer the latent processes that trigger, amplify, and sustain these events. By integrating data assimilation with information-theoretic and trajectory-based diagnostics, the framework reconstructs hidden precursor dynamics, quantifies their uncertainty, and determines how their temporal influence propagates toward observed extreme episodes.

Two complementary viewpoints were introduced. From a trajectory-wise perspective, filtering and smoothing hidden variable distributions were compared to identify the onset of hidden precursors and to quantify their temporal influence ranges. These tools provide event-specific attribution and reveal whether the mechanisms responsible for an extreme event are detectable in real time or only retrospectively after future observations become available. From a statistical perspective, event-conditioned hidden-state distributions were used to identify the most sensitive triggering directions, representative latent pathways, and coherent classes of extreme-event pathways.

Three prototype examples illustrated the methodology. In an intermittent stochastic model, hidden damping dynamics emerged before observed bursts, and the discrepancy between filtering and smoothing provided a natural onset diagnostic. In a model with multiple hidden variables, distinct forcing-driven, damping-induced, and mixed pathways to extreme events were identified through clustering and conditional statistics. In a nonlinear topographic-flow model, the framework revealed multiple hidden pathways associated with blocking and unblocking transitions, highlighting the relevance of the approach for more realistic geophysical systems.

\subsection{Discussion}

A central message of this work is that observed extreme events may arise through multiple distinct latent pathways, even when their visible signatures appear similar. In such situations, analyses based only on the observed variables can be incomplete or even misleading, because different hidden pathways may project onto nearly identical observable behavior. The numerical examples showed that events with similar amplitudes and temporal patterns can be generated by fundamentally different precursor dynamics. This highlights the importance of hidden variables, unresolved processes, and latent degrees of freedom in understanding extreme-event generation.

Identifying multiple pathways is also important for scientific interpretation. Distinct hidden mechanisms often correspond to different physical regimes, competing instabilities, or transitions between metastable states. Therefore, separating these pathways can improve our understanding of regime switching, tipping behavior, intermittency, and nonlinear instability growth. Such phenomena are central in Earth systems, including atmospheric blocking, oceanic variability, abrupt climate transitions, and ecosystem shifts \cite{boers2022theoretical, lucarini2024detecting}, but they also arise broadly in engineering systems involving turbulence, structural failure, combustion, power grids, and complex networks \cite{lin2014tipping}. A mechanism-based view of extremes can therefore complement purely statistical descriptions and provide deeper physical insight into why rare events occur.

Another important conclusion is the temporal characterization of hidden mechanisms through onset diagnostics and influence ranges. Beyond identifying which latent processes are associated with an extreme event, the proposed framework asks when these processes first become dynamically relevant and how long their effects persist. The onset time provides an estimate of when the hidden precursor begins, which may occur substantially earlier than any visible signal in the observed variables. The influence range then quantifies whether the precursor acts only as an initial trigger or continues to shape the growth, peak, and subsequent recovery phases of the event. Equally important, the comparison between filtering and smoothing reveals that real-time detection can fail when crucial precursor information remains concealed in unresolved variables. In such cases, online estimates based only on past observations may miss the true beginning of an extreme event, whereas retrospective smoothing can recover the latent pathway more accurately. This distinction is particularly relevant for early warning systems, where understanding the limits of real-time detectability is as important as improving prediction skill itself.

The results also have practical implications for model reduction, surrogate modeling, and scientific machine learning. In many applications, reduced-order or data-driven models are constructed only from observed variables. Such models may reproduce equilibrium statistics, marginal distributions, or autocorrelation functions with reasonable accuracy, yet still fail to predict extreme events if the true triggering factors reside in unresolved hidden processes. Missing latent mechanisms can lead to systematically incorrect precursor dynamics, poor early warning skill, and biased risk estimates. The present framework therefore motivates models that explicitly incorporate stochastic hidden variables or parameterized latent processes. Distinguishing forcing-driven, damping-driven, and mixed pathways further provides guidance for how these unresolved effects should be represented in reduced-order closures and stochastic parametrizations. More broadly, once the essential pathways responsible for extremes are identified in a general nonlinear system, one can construct analytically tractable surrogates that retain the key interactions while simplifying secondary details. In particular, this creates a natural pathway for developing conditional Gaussian surrogate models, as well as machine-learning architectures built upon the same structure, such as conditional Gaussian Koopman networks \cite{chen2025modeling}. These approaches combine physical interpretability with major computational advantages: filtering, smoothing, prediction, and uncertainty quantification can be performed efficiently through closed-form or highly scalable algorithms. This offers a promising route for translating mechanism-level understanding in complex systems into practical forecasting and risk-assessment tools.

Finally, one important lesson comes from the third case study in Section \ref{Sec:Numerics}: extreme events are not always associated with a specific deterministic location in the hidden-state space. Instead, they may be triggered by enhanced variability, broadened uncertainty, or noise-induced excursions across dynamically sensitive regions. In these settings, the relevant signature of extremes lies not only in the geometry of attractors or coherent trajectories, but also in changes in probability distributions and fluctuation structures. This suggests that the study of extreme events should combine dynamical-systems ideas with statistical descriptions of uncertainty, variability, and rare transitions. Understanding extremes therefore requires analyzing both deterministic pathways and stochastic structures in the state space.

\subsection{Future work}

Several directions remain open. A key challenge is to extend the present framework to high-dimensional systems with strongly nonlinear hidden dynamics using scalable ensemble, particle, and machine-learning-assisted smoothing methods. Another important direction for future work is the development of efficient mixture-model representations of extreme-event-conditioned distributions for general nonlinear systems \eqref{General_System} outside the conditional Gaussian class \eqref{CGNS}, potentially using importance-sampling methods. Applying these ideas to realistic Earth system, turbulent flow, and engineering problems may uncover practically relevant precursor mechanisms and improve early warning capabilities. It is also of interest to embed the identified hidden mechanisms directly into reduced-order, stochastic, and learning-based predictive models. More broadly, connecting mechanism discovery with forecasting, uncertainty quantification, and control of extreme events remains an important direction for future research.

\section*{Acknowledgment}
N.C. is supported by the Office of Naval Research (ONR) under Award No. N00014-24-1-2244 and by the Army Research Office (ARO) under Award No. W911NF-23-1-0118. C.M. is partially supported by the aforementioned ONR award and partially by the University of Wisconsin-Madison Office of the Vice Chancellor for Research, with funding from the Wisconsin Alumni Research Foundation. M.A. is partially supported by the aforementioned ARO award and partially by the National Science Foundation (Award No. DMS-2023239) through the Institute for Foundations of Data Science at the University of Wisconsin-Madison.

\section*{Conflict of Interest Statement}
The authors have no conflicts to disclose.

\section*{Author Contributions}

Charlotte Rose Moser: Formal analysis (equal); Investigation (equal); Methodology (equal); Validation (equal); Writing -- original draft (equal); Writing -- review and editing (equal).

Nan Chen: Conceptualization (equal); Data curation (equal); Formal analysis (equal); Funding acquisition (equal); Investigation (equal); Methodology (equal); Project administration (equal); Supervision (equal); Validation (equal); Writing -- original draft (equal).

Marios Andreou: Conceptualization (equal); Formal analysis (equal); Investigation (equal); Methodology (equal); Validation (equal); Writing -- review and editing (equal).

\section*{Data Availability Statement}
The data that support the findings of this study were generated through numerical simulations of the models described in the article. No external datasets were used.


\appendix

\section{Feature variables used for clustering extreme events in the stochastic model with both latent damping and forcing}

To distinguish multiple pathways leading to extreme events in the stochastic model with hidden damping and forcing variables in \eqref{SPEKF_MA_num}, we construct a set of physically interpretable feature variables for each detected event. These features summarize key aspects of the event evolution, including peak amplitude, duration, precursor behavior, and the relative contributions of damping and stochastic forcing during the growth phase. They are used in the clustering analysis of Section~\ref{Sec:Mechanisms} to separate extreme events into coherent groups associated with distinct dynamical pathways. The variables listed below are computed from the aligned event trajectories and, when appropriate, normalized before clustering. In the numerical experiments presented here, the clustering features are computed from the simulated true trajectories to characterize the underlying pathway classes without reconstruction error. Alternatively, these features could be constructed using the smoother-based estimates of the hidden variables, allowing the classification to rely on observed and reconstructed information rather than the simulated truth. Prior to clustering, each feature is standardized across events to prevent variables with larger numerical scales from dominating the distance metric.

For each detected extreme event with peak time $t_*$, a feature vector is constructed from the event amplitude, temporal shape, pre-peak observable behavior, hidden-state information, and pathway diagnostics. For computing pre-peak statistics, a window spanning the time steps from $t_* - 1.2$ to $t_*$ is used. The full feature set contains the following 18 quantities.

\begin{enumerate}
    \item {Signed peak amplitude:}
    \[
    u(t_*).
    \]
    This preserves the sign of the event and distinguishes positive and negative extremes.

    \item {Absolute peak amplitude:}
    \[
    |u(t_*)|.
    \]
    This measures the event intensity regardless of sign.

    \item {Event duration:}
    Defined as the temporal width during which the trajectory remains above one-half of the peak magnitude with the same sign. It distinguishes sharp bursts from persistent extremes.

    \item {Pre-peak slope of $u$:}
    Obtained from a linear regression of $u(t)$ over the pre-peak window. It quantifies the growth rate approaching the peak.

    \item {Pre-peak mean of $u$:}
    The average state before the peak, indicating whether the event emerges from an already elevated background.

    \item {Pre-peak mean of $\gamma$:}
    Measures the average hidden damping/growth environment before the event.

    \item {Pre-peak mean of $b$:}
    Measures the average stochastic forcing background before the event.

    \item {Pre-peak mean of the damping/growth term:}
    \[
    G(t)=(-d_u+c\gamma(t))u(t).
    \]
    This quantifies the average multiplicative contribution from the damping/growth pathway.

    \item {Pre-peak mean of the forcing term:}
    \[
    B(t)=b(t).
    \]
    This quantifies the average additive forcing contribution.

    \item {Integrated damping/growth contribution:}
    \[
    \int_{\mathrm{pre}} G(t)\,\mathrm{d}t.
    \]
    Measures the cumulative effect of the damping/growth pathway before the peak.

    \item {Integrated forcing contribution:}
    \[
    \int_{\mathrm{pre}} B(t)\,\mathrm{d}t.
    \]
    Measures the cumulative effect of the forcing pathway before the peak.

    \item {Signed integrated damping/growth contribution:}
    The cumulative damping/growth contribution multiplied by the sign of the peak event. Positive values indicate that this pathway supports the eventual extreme.

    \item {Signed integrated forcing contribution:}
    The cumulative forcing contribution multiplied by the sign of the peak event. Positive values indicate forcing aligned with the final event sign.

    \item {Relative damping/growth importance:}
    \[
    \frac{\int_{\mathrm{pre}} |G(t)|\,dt}
    {\int_{\mathrm{pre}} |G(t)|\,dt+\int_{\mathrm{pre}} |B(t)|\,dt}.
    \]
    This measures the fraction of total pathway activity attributable to the damping/growth mechanism.

    \item {Relative forcing importance:}
    \[
    \frac{\int_{\mathrm{pre}} |B(t)|\,dt}
    {\int_{\mathrm{pre}} |G(t)|\,dt+\int_{\mathrm{pre}} |B(t)|\,dt}.
    \]
    This measures the fraction of total pathway activity attributable to stochastic forcing.

    \item {State one unit before peak:}
    \[
    u(t_*-1).
    \]
    Captures the observable precursor level shortly before the event.

    \item {Hidden damping/growth state one unit before peak:}
    \[
    \gamma(t_*-1).
    \]
    Captures the near-onset hidden damping/growth condition.

    \item {Hidden forcing state one unit before peak:}
    \[
    b(t_*-1).
    \]
    Captures the near-onset forcing condition.
\end{enumerate}

These 18 features jointly encode event magnitude, duration, temporal growth, hidden precursors, cumulative mechanism strength, and relative pathway dominance. After standardization, they provide a rich representation for separating extreme events generated by distinct routes such as damping-driven, forcing-driven, and mixed-interaction pathways.

\section{Feature variables used for clustering topographic-flow extreme events}

For the topographic-flow model in \eqref{eq:topo_fourier_full}, extreme events may arise through different regime-transition pathways and involve distinct precursor structures in the resolved and unresolved modes. To classify these events systematically, we introduce a collection of feature variables that characterize the amplitude, temporal evolution, energetic growth, and modal interactions associated with each event. These quantities provide a reduced but interpretable description of the underlying dynamics and are used in the clustering procedure described in Section~\ref{Sec:Mechanisms}, with features likewise computed from the simulated true trajectories. Prior to clustering the same pre-processing procedure is carried out as described in Appendix A. The feature definitions are summarized below.

For each detected extreme event, we construct a feature vector from the peak amplitude, precursor evolution, and hidden-mode energy distribution. The pre-peak, peak, and post-peak windows used to compute the relevant statistics are $[t_*-1, t_*-0.3]$, $[t_*-0.1, t_*+0.1]$, and $[t_*+0.2, t_* + 1]$. Several of the features describe the modal energy $E=\sum_{\mathbf{k}}|\phi_{\mathbf{k}}|^2$. The following quantities are used.

\begin{enumerate}
\item $V_{\mathrm{peak}}$: peak value of $V$ at the event time.
\item $|V_{\mathrm{peak}}|$: absolute peak amplitude.
\item $S_{\mathrm{pre}}^{(s)}$: short-window pre-event growth rate of $V$.
\item $S_{\mathrm{pre}}^{(l)}$: long-window pre-event growth rate of $V$.
\item $S_{\mathrm{post}}$: short-window post-event decay rate of $V$.
\item $\mathrm{Std}_{\mathrm{pre}}(V)$: local pre-event standard deviation of $V$.
\item $\mathrm{Mean}_{\mathrm{pre}}(V)$: pre-event mean of $V$.
\item $\mathrm{Mean}_{\mathrm{post}}(V)$: post-event mean of $V$.
\item $E_{\mathrm{tot}}^{\mathrm{peak}}$: total modal energy at the event time.
\item $E_{\mathrm{tot}}^{\mathrm{pre,mean}}$: mean total modal energy in the precursor window.
\item $E_{\mathrm{tot}}^{\mathrm{pre,max}}$: maximum total modal energy in the precursor window.
\item $E_{\mathrm{upper}}$: precursor energy in upper-diagonal modes $(1,1),(2,1)$.
\item $E_{\mathrm{zonal}}$: precursor energy in zonal modes $(1,0),(2,0)$.
\item $E_{\mathrm{lower}}$: precursor energy in lower-diagonal modes $(1,-1),(2,-1)$.
\item $E_{\mathrm{merid}}$: precursor energy in meridional modes $(0,1),(0,2)$.
\item $E_{\mathrm{upper}}^{\max}$: maximum upper-diagonal energy before the event.
\item $E_{\mathrm{zonal}}^{\max}$: maximum zonal energy before the event.
\item $E_{\mathrm{lower}}^{\max}$: maximum lower-diagonal energy before the event.
\item $R_{1}=E_{\mathrm{upper}}/E_{\mathrm{lower}}$: diagonal asymmetry ratio.
\item $R_{2}=E_{\mathrm{zonal}}/(E_{\mathrm{upper}}+E_{\mathrm{lower}})$: zonal-to-diagonal ratio.
\item $R_{3}=(E_{\mathrm{upper}}-E_{\mathrm{lower}})/(E_{\mathrm{upper}}+E_{\mathrm{lower}})$: signed diagonal imbalance.
\item Mean precursor $\Re(\phi_{(1,1)})$.
\item Mean precursor $\Re(\phi_{(2,1)})$.
\item Mean precursor $\Re(\phi_{(1,0)})$.
\item Mean precursor $\Re(\phi_{(1,-1)})$.
\item Mean precursor $\Re(\phi_{(2,-1)})$.
\item Peak-time $\Re(\phi_{(1,1)})$.
\item Peak-time $\Re(\phi_{(1,0)})$.
\item Peak-time $\Re(\phi_{(1,-1)})$.
\end{enumerate}

These features combine observable event shape, hidden-mode energy partition, anisotropy, and representative modal amplitudes, allowing different dynamical pathways of extreme unblocking to be separated through unsupervised clustering.


%
%

%


\bibliography{references}

@article{kaveh2025spatiotemporal,
  title={Spatiotemporal forecast of extreme events in a chaotic model of slow slip events},
  author={Kaveh, Hojjat and Avouac, Jean Philippe and Stuart, Andrew M},
  journal={Geophysical Journal International},
  volume={240},
  number={2},
  pages={870--885},
  year={2025},
  publisher={Oxford University Press}
}

@article{farazmand2019extreme,
  title={Extreme events: {M}echanisms and prediction},
  author={Farazmand, Mohammad and Sapsis, Themistoklis P},
  journal={Applied Mechanics Reviews},
  volume={71},
  number={5},
  pages={050801},
  year={2019},
  publisher={American Society of Mechanical Engineers}
}

@article{guth2019machine,
  title={Machine learning predictors of extreme events occurring in complex dynamical systems},
  author={Guth, Stephen and Sapsis, Themistoklis P},
  journal={Entropy},
  volume={21},
  number={10},
  pages={925},
  year={2019},
  publisher={MDPI}
}

@article{chen2020predicting,
  title={Predicting observed and hidden extreme events in complex nonlinear dynamical systems with partial observations and short training time series},
  author={Chen, Nan and Majda, Andrew J},
  journal={Chaos: An Interdisciplinary Journal of Nonlinear Science},
  volume={30},
  number={3},
  year={2020},
  publisher={AIP Publishing}
}

@article{farazmand2017variational,
  title={A variational approach to probing extreme events in turbulent dynamical systems},
  author={Farazmand, Mohammad and Sapsis, Themistoklis P},
  journal={Science advances},
  volume={3},
  number={9},
  pages={e1701533},
  year={2017},
  publisher={American Association for the Advancement of Science}
}

@article{ghil2011extreme,
  title={Extreme events: dynamics, statistics and prediction},
  author={Ghil, M and Yiou, Pascal and Hallegatte, St{\'e}phane and Malamud, BD and Naveau, P and Soloviev, A and Friederichs, P and Keilis-Borok, V and Kondrashov, D and Kossobokov, V and others},
  journal={Nonlinear Processes in Geophysics},
  volume={18},
  number={3},
  pages={295--350},
  year={2011},
  publisher={Copernicus GmbH}
}

@article{sapsis2021statistics,
  title={Statistics of extreme events in fluid flows and waves},
  author={Sapsis, Themistoklis P},
  journal={Annual Review of Fluid Mechanics},
  volume={53},
  number={1},
  pages={85--111},
  year={2021},
  publisher={Annual Reviews}
}

@article{trenberth2015attribution,
  title={Attribution of climate extreme events},
  author={Trenberth, Kevin E and Fasullo, John T and Shepherd, Theodore G},
  journal={Nature climate change},
  volume={5},
  number={8},
  pages={725--730},
  year={2015},
  publisher={Nature Publishing Group UK London}
}

@article{mohamad2018sequential,
  title={Sequential sampling strategy for extreme event statistics in nonlinear dynamical systems},
  author={Mohamad, Mustafa A and Sapsis, Themistoklis P},
  journal={Proceedings of the National Academy of Sciences},
  volume={115},
  number={44},
  pages={11138--11143},
  year={2018},
  publisher={National Academy of Sciences}
}

@article{mackay2021sampling,
  title={Sampling properties and empirical estimates of extreme events},
  author={Mackay, Ed and Jonathan, Philip},
  journal={Ocean Engineering},
  volume={239},
  pages={109791},
  year={2021},
  publisher={Elsevier}
}

@article{webber2019practical,
  title={Practical rare event sampling for extreme mesoscale weather},
  author={Webber, Robert J and Plotkin, David A and O’Neill, Morgan E and Abbot, Dorian S and Weare, Jonathan},
  journal={Chaos: An Interdisciplinary Journal of Nonlinear Science},
  volume={29},
  number={5},
  year={2019},
  publisher={AIP Publishing}
}

@article{finkel2024bringing,
  title={Bringing statistics to storylines: {R}are event sampling for sudden, transient extreme events},
  author={Finkel, Justin and O’Gorman, Paul A},
  journal={Journal of Advances in Modeling Earth Systems},
  volume={16},
  number={6},
  pages={e2024MS004264},
  year={2024},
  publisher={Wiley Online Library}
}

@article{sun2025can,
  title={Can {AI} weather models predict out-of-distribution gray swan tropical cyclones?},
  author={Sun, Y Qiang and Hassanzadeh, Pedram and Zand, Mohsen and Chattopadhyay, Ashesh and Weare, Jonathan and Abbot, Dorian S},
  journal={Proceedings of the National Academy of Sciences},
  volume={122},
  number={21},
  pages={e2420914122},
  year={2025},
  publisher={National Academy of Sciences}
}

@article{mojgani2023extreme,
  title={Extreme event prediction with multi-agent reinforcement learning-based parametrization of atmospheric and oceanic turbulence},
  author={Mojgani, Rambod and Waelchli, Daniel and Guan, Yifei and Koumoutsakos, Petros and Hassanzadeh, Pedram},
  journal={arXiv preprint arXiv:2312.00907},
  year={2023}
}

@article{ansmann2013extreme,
  title={Extreme events in excitable systems and mechanisms of their generation},
  author={Ansmann, Gerrit and Karnatak, Rajat and Lehnertz, Klaus and Feudel, Ulrike},
  journal={Physical Review E—Statistical, Nonlinear, and Soft Matter Physics},
  volume={88},
  number={5},
  pages={052911},
  year={2013},
  publisher={APS}
}

@article{noy2024extreme,
  title={Extreme events impact attribution: a state of the art},
  author={Noy, Ilan and Stone, D{\'a}ith{\'\i} and Uher, Tom{\'a}{\v{s}}},
  journal={Cell Reports Sustainability},
  volume={1},
  number={5},
  year={2024},
  publisher={Elsevier}
}

@article{zhang2025physics,
  title={Physics-assisted data-driven stochastic reduced-order models for attribution of heterogeneous stress distributions in low-grain polycrystals},
  author={Zhang, Yinling and Dunham, Samuel D and Bronkhorst, Curt A and Chen, Nan},
  journal={Proceedings of the Royal Society A},
  volume={481},
  number={2309},
  pages={20240898},
  year={2025},
  publisher={The Royal Society}
}

@article{akhmediev2016roadmap,
  title={Roadmap on optical rogue waves and extreme events},
  author={Akhmediev, Nail and Kibler, Bertrand and Baronio, Fabio and Beli{\'c}, Milivoj and Zhong, Wei-Ping and Zhang, Yiqi and Chang, Wonkeun and Soto-Crespo, Jose M and Vouzas, Peter and Grelu, Philippe and others},
  journal={Journal of Optics},
  volume={18},
  number={6},
  pages={063001},
  year={2016},
  publisher={IOP Publishing}
}

@article{phadnis2021study,
  title={A study of the effect of {B}lack {S}wan events on stock markets--and developing a model for predicting and responding to them},
  author={Phadnis, Chinmay and Joshi, Sunit and Sharma, Dipasha},
  journal={Australasian Accounting, Business and Finance Journal},
  volume={15},
  number={1},
  year={2021}
}

@article{santoso2017defining,
  title={The defining characteristics of {ENSO} extremes and the strong 2015/2016 {E}l {N}i{\~n}o},
  author={Santoso, Agus and Mcphaden, Michael J and Cai, Wenju},
  journal={Reviews of Geophysics},
  volume={55},
  number={4},
  pages={1079--1129},
  year={2017},
  publisher={Wiley Online Library}
}

@article{chen2018conditional,
  title={Conditional {G}aussian systems for multiscale nonlinear stochastic systems: {P}rediction, state estimation and uncertainty quantification},
  author={Chen, Nan and Majda, Andrew J},
  journal={Entropy},
  volume={20},
  number={7},
  pages={509},
  year={2018},
  publisher={MDPI}
}

@book{liptser2001statistics,
  title={Statistics of Random Processes I, II},
  publisher={Springer Berlin Heidelberg},
  author={Liptser, Robert Shevilevich and Shiryaev, Albert Nikolaevich},
  volume={1, 2},
  year={2001}
}

@article{kalman1961new,
  title={New results in linear filtering and prediction theory},
  author={Kalman, Rudolph E and Bucy, Richard S},
  journal={Journal of Basic Engineering},
  volume={83},
  number={1},
  pages={95--108},
  year={1961}
}

@book{sarkka2023bayesian,
  title={Bayesian filtering and smoothing},
  author={S{\"a}rkk{\"a}, Simo and Svensson, Lennart},
  volume={17},
  year={2023},
  publisher={Cambridge university press}
}

@article{guan2026prediction,
  title={Prediction of Extreme Events in Multiscale Simulations of Geophysical Turbulence using Reinforcement Learning},
  author={Guan, Yifei and Amoudruz, Lucas and Litvinov, Sergey and Jakhar, Karan and Mojgani, Rambod and Koumoutsakos, Petros and Hassanzadeh, Pedram},
  journal={arXiv preprint arXiv:2603.03351},
  year={2026}
}

@book{williams1991probability,
  title={Probability with martingales},
  author={Williams, David},
  year={1991},
  publisher={Cambridge university press}
}

@article{kullback1951information,
  title={On information and sufficiency},
  author={Kullback, Solomon and Leibler, Richard A},
  journal={The annals of mathematical statistics},
  volume={22},
  number={1},
  pages={79--86},
  year={1951},
  publisher={JSTOR}
}

@book{kullback1997information,
  title={Information theory and statistics},
  author={Kullback, Solomon},
  year={1997},
  publisher={Courier Corporation}
}

@article{chowdhury2022extreme,
  title={Extreme events in dynamical systems and random walkers: {A} review},
  author={Chowdhury, Sayantan Nag and Ray, Arnob and Dana, Syamal K and Ghosh, Dibakar},
  journal={Physics Reports},
  volume={966},
  pages={1--52},
  year={2022},
  publisher={Elsevier}
}

@article{alvre2024studying,
  title={Studying extreme events: An interdisciplinary review of recent research},
  author={Alvre, J and Broska, LH and R{\"u}bbelke, DTG and V{\"o}gele, S},
  journal={Heliyon},
  volume={10},
  number={24},
  year={2024},
  publisher={Elsevier}
}

@article{chang2025extreme,
  title={Extreme Event Aware ($\eta$-)Learning},
  author={Chang, Kai and Sapsis, Themistoklis P},
  journal={arXiv preprint arXiv:2510.19161},
  year={2025}
}

@article{mishra2020routes,
  title={Routes to extreme events in dynamical systems: {D}ynamical and statistical characteristics},
  author={Mishra, Arindam and Leo Kingston, S and Hens, Chittaranjan and Kapitaniak, Tomasz and Feudel, Ulrike and Dana, Syamal K},
  journal={Chaos: An Interdisciplinary Journal of Nonlinear Science},
  volume={30},
  number={6},
  year={2020},
  publisher={AIP Publishing}
}

@book{chen2023stochastic,
  title={Stochastic methods for modeling and predicting complex dynamical systems},
  author={Chen, Nan},
  year={2023},
  publisher={Springer}
}

@article{majda2018model,
  title={Model error, information barriers, state estimation and prediction in complex multiscale systems},
  author={Majda, Andrew J and Chen, Nan},
  journal={Entropy},
  volume={20},
  number={9},
  pages={644},
  year={2018},
  publisher={MDPI}
}

@article{grigorian2025learning,
  title={Learning governing equations of unobserved states in dynamical systems},
  author={Grigorian, Gevik and George, Sandip V and Arridge, Simon},
  journal={Physica D: Nonlinear Phenomena},
  volume={472},
  pages={134499},
  year={2025},
  publisher={Elsevier}
}

@article{boyen2013discovering,
  title={Discovering the hidden structure of complex dynamic systems},
  author={Boyen, Xavier and Friedman, Nir and Koller, Daphne},
  journal={arXiv preprint arXiv:1301.6683},
  year={2013}
}

@article{lovejoy2018spectra,
  title={Spectra, intermittency, and extremes of weather, macroweather and climate},
  author={Lovejoy, S},
  journal={Scientific reports},
  volume={8},
  number={1},
  pages={12697},
  year={2018},
  publisher={Nature Publishing Group UK London}
}

@article{overland2021intermittency,
  title={How do intermittency and simultaneous processes obfuscate the {A}rctic influence on midlatitude winter extreme weather events?},
  author={Overland, James E and Ballinger, Thomas J and Cohen, Judah and Francis, JA and Hanna, Edward and Jaiser, Ralf and Kim, B-M and Kim, S-J and Ukita, Jinro and Vihma, Timo and others},
  journal={Environmental Research Letters},
  volume={16},
  number={4},
  pages={043002},
  year={2021},
  publisher={IOP Publishing}
}

@article{finkel2026rare,
  title={Rare event sampling for moving targets: {E}xtremes of temperature and daily precipitation in a general circulation model},
  author={Finkel, Justin and O’Gorman, Paul A},
  journal={Journal of Advances in Modeling Earth Systems},
  volume={18},
  number={3},
  pages={e2025MS005456},
  year={2026},
  publisher={Wiley Online Library}
}

@book{asch2016data,
  title={Data assimilation: methods, algorithms, and applications},
  author={Asch, Mark and Bocquet, Marc and Nodet, Ma{\"e}lle},
  year={2016},
  publisher={SIAM}
}

@book{law2015data,
  title={Data Assimilation: A Mathematical Introduction},
  author={Law, Kody and Stuart, Andrew and Zygalakis, Kostas},
  series={Texts in Applied Mathematics},
  volume={62},
  year={2015},
  publisher={Springer}
}

@article{andreou2026assimilative,
  title={Assimilative causal inference},
  author={Andreou, Marios and Chen, Nan and Bollt, Erik},
  journal={Nature Communications},
  volume={17},
  pages={1854},
  year={2026},
  publisher={Nature Publishing Group}
}

@article{cohn1994fixed,
  title={A fixed-lag {K}alman smoother for retrospective data assimilation},
  author={Cohn, Stephen E and Sivakumaran, NS and Todling, Ricardo},
  journal={Monthly Weather Review},
  volume={122},
  number={12},
  pages={2838--2867},
  year={1994}
}

@article{alexander2005accelerated,
  title={Accelerated {M}onte {C}arlo for optimal estimation of time series},
  author={Alexander, Francis J and Eyink, Gregory L and Restrepo, Juan M},
  journal={Journal of Statistical Physics},
  volume={119},
  number={5},
  pages={1331--1345},
  year={2005},
  publisher={Springer}
}

@book{evensen2009data,
  title={Data assimilation: {T}he ensemble {K}alman filter},
  author={Evensen, Geir},
  year={2009},
  publisher={Springer}
}

@article{branicki2018accuracy,
  title={Accuracy of some approximate {G}aussian filters for the {N}avier--{S}tokes equation in the presence of model error},
  author={Branicki, Michal and Majda, Andrew J and Law, Kody JH},
  journal={Multiscale Modeling \& Simulation},
  volume={16},
  number={4},
  pages={1756--1794},
  year={2018},
  publisher={SIAM}
}

@article{harlim2010filtering,
  title={Filtering turbulent sparsely observed geophysical flows},
  author={Harlim, John and Majda, Andrew J},
  journal={Monthly Weather Review},
  volume={138},
  number={4},
  pages={1050--1083},
  year={2010}
}

@article{keating2012new,
  title={New methods for estimating ocean eddy heat transport using satellite altimetry},
  author={Keating, Shane R and Majda, Andrew J and Smith, K Shafer},
  journal={Monthly Weather Review},
  volume={140},
  number={5},
  pages={1703--1722},
  year={2012}
}

@article{branicki2013non,
  title={Non-{G}aussian test models for prediction and state estimation with model errors},
  author={Branicki, Michal and Chen, Nan and Majda, Andrew J},
  journal={Chinese Annals of Mathematics, Series B},
  volume={34},
  number={1},
  pages={29--64},
  year={2013},
  publisher={Springer}
}

@article{majda2013physics,
  title={Physics constrained nonlinear regression models for time series},
  author={Majda, Andrew J and Harlim, John},
  journal={Nonlinearity},
  volume={26},
  number={1},
  pages={201--217},
  year={2013},
  publisher={IOP Publishing}
}

@article{brunner2017connecting,
  title={Connecting atmospheric blocking to {E}uropean temperature extremes in spring},
  author={Brunner, Lukas and Hegerl, Gabriele C and Steiner, Andrea K},
  journal={Journal of Climate},
  volume={30},
  number={2},
  pages={585--594},
  year={2017}
}

@book{vallis2017atmospheric,
  title={Atmospheric and oceanic fluid dynamics},
  author={Vallis, Geoffrey K},
  year={2017},
  publisher={Cambridge University Press}
}

@article{qi2018predicting,
  title={Predicting extreme events for passive scalar turbulence in two-layer baroclinic flows through reduced-order stochastic models},
  author={Qi, Di and Majda, Andrew J},
  journal={Commun. Math. Sci},
  volume={16},
  number={1},
  pages={17--51},
  year={2018}
}

@book{majda2006nonlinear,
  title={Nonlinear dynamics and statistical theories for basic geophysical flows},
  author={Majda, Andrew and Wang, Xiaoming},
  year={2006},
  publisher={Cambridge University Press}
}

@article{chen2024physics,
  title={A physics-informed data-driven algorithm for ensemble forecast of complex turbulent systems},
  author={Chen, Nan and Qi, Di},
  journal={Applied Mathematics and Computation},
  volume={466},
  pages={128480},
  year={2024},
  publisher={Elsevier}
}

@article{chen2025modeling,
  title={Modeling partially observed nonlinear dynamical systems and efficient data assimilation via discrete-time conditional {G}aussian {K}oopman network},
  author={Chen, Chuanqi and Wang, Zhongrui and Chen, Nan and Wu, Jin-Long},
  journal={Computer Methods in Applied Mechanics and Engineering},
  volume={445},
  pages={118189},
  year={2025},
  publisher={Elsevier}
}

@article{boers2022theoretical,
  title={Theoretical and paleoclimatic evidence for abrupt transitions in the {E}arth system},
  author={Boers, Niklas and Ghil, Michael and Stocker, Thomas F},
  journal={Environmental Research Letters},
  volume={17},
  number={9},
  pages={093006},
  year={2022},
  publisher={IOP Publishing}
}

@article{lucarini2024detecting,
  title={Detecting and attributing change in climate and complex systems: {F}oundations, {G}reen's functions, and nonlinear fingerprints},
  author={Lucarini, Valerio and Chekroun, Micka{\"e}l D},
  journal={Physical Review Letters},
  volume={133},
  number={24},
  pages={244201},
  year={2024},
  publisher={APS}
}

@article{lin2014tipping,
  title={Tipping points in seaweed genetic engineering: scaling up opportunities in the next decade},
  author={Lin, Hanzhi and Qin, Song},
  journal={Marine drugs},
  volume={12},
  number={5},
  pages={3025--3045},
  year={2014},
  publisher={MDPI}
}

@book{coles2001introduction,
  title={An introduction to statistical modeling of extreme values},
  author={Coles, Stuart},
  volume={208},
  year={2001},
  publisher={Springer}
}

@book{albeverio2006extreme,
  title={Extreme events in nature and society},
  author={Albeverio, Sergio and Jentsch, Volker and Kantz, Holger},
  year={2006},
  publisher={Springer Science \& Business Media}
}

@book{kharif2008rogue,
  title={Rogue waves in the ocean},
  author={Kharif, Christian and Pelinovsky, Efim and Slunyaev, Alexey},
  year={2008},
  publisher={Springer Science \& Business Media}
}

@book{lucarini2016extremes,
  title={Extremes and recurrence in dynamical systems},
  author={Lucarini, Valerio and Faranda, Davide and de Freitas, Jorge Miguel Milhazes and Holland, Mark and Kuna, Tobias and Nicol, Matthew and Todd, Mike and Vaienti, Sandro and others},
  year={2016},
  publisher={John Wiley \& Sons}
}

@book{majda2005information,
  title={Information theory and stochastics for multiscale nonlinear systems},
  author={Majda, Andrew and Abramov, Rafail V and Grote, Marcus J},
  volume={25},
  year={2005},
  publisher={American Mathematical Soc.}
}

@book{crisan2011oxford,
  title={The Oxford handbook of nonlinear filtering},
  author={Crisan, Dan and Rozovskii, Boris},
  year={2011},
  publisher={Oxford university press}
}

@book{rozovskii2012stochastic,
  title={Stochastic evolution systems: linear theory and applications to non-linear filtering},
  author={Rozovskii, Boris Lvovich},
  year={2012},
  publisher={Springer Science \& Business Media}
}

@article{chen2020efficient,
  title={Efficient nonlinear optimal smoothing and sampling algorithms for complex turbulent nonlinear dynamical systems with partial observations},
  author={Chen, Nan and Majda, Andrew J},
  journal={Journal of Computational Physics},
  volume={410},
  pages={109381},
  year={2020},
  publisher={Elsevier}
}

@article{chen2018efficient,
  title="{Efficient statistically accurate algorithms for the Fokker--Planck equation in large dimensions}",
  author={Chen, Nan and Majda, Andrew J},
  journal={Journal of Computational Physics},
  volume={354},
  pages={242--268},
  year={2018},
  publisher={Elsevier}
}

@article{kleeman2011information,
  title={Information theory and dynamical system predictability},
  author={Kleeman, Richard},
  journal={Entropy},
  volume={13},
  number={3},
  pages={612--649},
  year={2011},
  publisher={MDPI}
}

@incollection{johansen2008tutorial,
  title={A tutorial on particle filtering and smoothing: Fifteen years later},
  author={Doucet, Arnaud and Johansen, Adam M.},
  booktitle={The Oxford Handbook of Nonlinear Filtering},
  editor={Crisan, Dan and Rozovskii, Boris},
  pages={656--704},
  year={2011},
  publisher={Oxford University Press}
}

@article{taghvaei2023survey,
  title={A survey of feedback particle filter and related controlled interacting particle systems (CIPS)},
  author={Taghvaei, Amirhossein and Mehta, Prashant G},
  journal={Annual Reviews in Control},
  volume={55},
  pages={356--378},
  year={2023},
  publisher={Elsevier}
}

@article{todling1998suboptimal,
  title={Suboptimal schemes for retrospective data assimilation based on the fixed-lag Kalman smoother},
  author={Todling, Ricardo and Cohn, Stephen E and Sivakumaran, NS},
  journal={Monthly Weather Review},
  volume={126},
  number={8},
  pages={2274--2286},
  year={1998}
}

@article{brovkin2021past,
  title={Past abrupt changes, tipping points and cascading impacts in the Earth system},
  author={Brovkin, Victor and Brook, Edward and Williams, John W and Bathiany, Sebastian and Lenton, Timothy M and Barton, Michael and DeConto, Robert M and Donges, Jonathan F and Ganopolski, Andrey and McManus, Jerry and others},
  journal={Nature Geoscience},
  volume={14},
  number={8},
  pages={550--558},
  year={2021},
  publisher={Nature Publishing Group UK London}
}

@article{kuehn2011mathematical,
  title={A mathematical framework for critical transitions: Bifurcations, fast--slow systems and stochastic dynamics},
  author={Kuehn, Christian},
  journal={Physica D: Nonlinear Phenomena},
  volume={240},
  number={12},
  pages={1020--1035},
  year={2011},
  publisher={Elsevier}
}

@article{thomas2016using,
  title={Using natural archives to detect climate and environmental tipping points in the Earth System},
  author={Thomas, Zoe A},
  journal={Quaternary Science Reviews},
  volume={152},
  pages={60--71},
  year={2016},
  publisher={Elsevier}
}

@book{boyd2004convex,
  title="{Convex Optimization}",
  author={Boyd, Stephen and Vandenberghe, Lieven},
  year={2004},
  publisher={Cambridge University Press}
}

@book{nesterov2018lectures,
  title="{Lectures on Convex Optimization}",
  author={Nesterov, Yurii},
  volume={137},
  year={2018},
  publisher={Springer}
}

@article{andreou2025bridging,
  title={Bridging prediction and attribution: Identifying forward and backward causal influence ranges using assimilative causal inference},
  author={Andreou, Marios and Chen, Nan},
  journal={Physica D: Nonlinear Phenomena},
  volume={498},
  pages={135405},
  year={2026},
  doi={10.1016/j.physd.2026.135405}
}

@inproceedings{kurkoski2009single,
  title = {Single-Gaussian messages and noise thresholds for decoding low-density lattice codes},
  url = {http://dx.doi.org/10.1109/ISIT.2009.5205680},
  DOI = {10.1109/isit.2009.5205680},
  booktitle = {2009 IEEE International Symposium on Information Theory},
  publisher = {IEEE},
  author = {Kurkoski,  Brian M. and Yamaguchi,  Kazuhiko and Kobayashi,  Kingo},
  year = {2009},
  month = jun,
  pages = {734–738}
}

@article{minka2013expectation,
  title="{Expectation propagation for approximate Bayesian inference}",
  author={Minka, Thomas P},
  journal={arXiv preprint arXiv:1301.2294},
  year={2013}
}

@article{majda2012lessons,
  title = {Lessons in uncertainty quantification for turbulent dynamical systems},
  volume = {32},
  ISSN = {1553-5231},
  url = {http://dx.doi.org/10.3934/dcds.2012.32.3133},
  DOI = {10.3934/dcds.2012.32.3133},
  number = {9},
  journal = {Discrete and Continuous Dynamical Systems},
  publisher = {American Institute of Mathematical Sciences (AIMS)},
  author = {Majda,  Andrew J. and Branicki,  Michal},
  year = {2012},
  pages = {3133–3221}
}

\end{document}